\documentclass{amsart}
\usepackage{amsfonts}
\usepackage{amsmath}
\usepackage{amssymb}
\usepackage{amsthm}
\usepackage{yhmath}
\usepackage[shortlabels]{enumitem}
\usepackage{cite}
\usepackage{bm}
\usepackage{mathrsfs}
\usepackage{euscript}
\usepackage{mathtools}
\usepackage{tensor}
\usepackage{tikz-cd}
\usetikzlibrary{decorations.markings,intersections}
\usepackage{upgreek}

\newcommand{\opn}{\operatorname}

\newcommand{\AND}{\mathrm{\; and \;}}

\newcommand{\sps}[1]{\left( #1 \right)}

\newcommand{\bk}[1]{\left\{ #1 \right\}}

\newcommand{\n}[1]{\lVert #1 \rVert}

\renewcommand{\a}[1]{\lvert #1 \rvert}

\newcommand{\eval}[1]{\left. #1 \right\rvert}
\ifdefined\stretchybydefault

\else

\fi

\newcommand{\tn}[1]{{\left\vert\kern-0.25ex\left\vert\kern-0.25ex\left\vert #1 
    \right\vert\kern-0.25ex\right\vert\kern-0.25ex\right\vert}}

\let\da\downarrow

\newcommand{\Lawvere}{{\mathrm{Lawvere}}}
\newcommand{\Linton}{{\mathrm{Linton}}}

\DeclareMathOperator{\id}{id}

\DeclareMathOperator{\cod}{cod}

\DeclareMathOperator{\ev}{ev}
\DeclareMathOperator{\ins}{ins}

\DeclareMathOperator{\Ob}{Ob}
\DeclareMathOperator{\Arr}{Arr}

\newcommand{\sint}{\smallint}

\newcommand{\R}{\mathbb{R}}

\newcommand{\N}{\mathbb{N}}

\let\ms\mathscr

\newcommand{\mc}[1]{\mathcal{#1}}

\providecommand{\comment}[1]{}

\newcommand{\cp}{\circ}

\newcommand{\into}{\hookrightarrow}

\renewcommand{\bar}{\overline}
\newcommand{\bbar}[1]{\overline{\overline{#1}}{}}

\let\al\alpha
\let\b\beta

\renewcommand{\k}{\kappa}
\let\d\delta
\let\e\epsilon
\let\io\iota

\let\Si\Sigma
\let\si\sigma
\let\ti\widetilde
\let\hat\widehat
\let\ld\lambda
\let\lp\Delta

\let\G\Gamma
\let\cd\cdot

\newcommand{\sli}{\mathcal{L}^\infty}

\providecommand{\ot}{\otimes}
\newcommand{\brot}{\mathbin{\bar{\otimes}}}

\newcommand{\tm}{\times}

\newcommand{\thnorm}[1]{|\kern-1pt|\kern-1pt| #1 |\kern-1pt|\kern-1pt|}
\theoremstyle{remark}
\newtheorem*{remark*}{Remark}
\theoremstyle{definition}
\newtheorem*{openproblem*}{Open Problem}
\newtheorem*{dfn*}{Definition}

\newtheorem*{example*}{Example}

\newtheorem{definition}{Definition}[section]
\newtheorem*{definition*}{Definition}
\theoremstyle{plain}
\newtheorem*{theorem*}{Theorem}
\newtheorem*{proposition*}{Proposition}
\newtheorem*{conjecture*}{Conjecture}
\newtheorem{theorem}{Theorem}[section]
\newtheorem{proposition}[theorem]{Proposition}

\newtheorem*{lemma*}{Lemma}
\newtheorem{lemma}[theorem]{Lemma}
\newtheorem*{corollary*}{Corollary}
\newtheorem{corollary}[theorem]{Corollary}
\theoremstyle{definition}

\theoremstyle{remark}

\numberwithin{subcase}{case}
\makeatletter
\@addtoreset{case}{theorem}
\@addtoreset{case}{proposition}
\@addtoreset{case}{lemma}
\@addtoreset{case}{corollary}
\makeatother

\newcommand{\qq}{\quad\quad}

\binoppenalty=\maxdimen
\relpenalty=\maxdimen

\makeatletter
\AtBeginDocument{\@ifpackageloaded{tikz}
  {\tikzset{commutative diagrams/.cd,
   mysymbol/.style={start anchor=center,end anchor=center,draw=none},
   sho/.style={shorten >=#1,shorten <=#1},
   sca/.style n args={3}{
      decoration={
    markings,
    mark=at position (1-#1)/2*\pgfdecoratedpathlength
      with {\coordinate (#2);},
    mark=at position (1+#1)/2*\pgfdecoratedpathlength
      with {\coordinate (#3);},
    },
  postaction=decorate,
  },
  raa/.style={arrows=Rightarrow}
}}
{\relax}}
\makeatletter
\newcommand\car[3][]{%
  \arrow[mysymbol]{#2}[description, #1]{#3}}
\newcommand\cart[1]{\car{#1}{\opn{cart}}}
\newcommand\opcart[1]{\car{#1}{\opn{trac}}}
\newcommand\armd{\ar[mapsto]{d}}
\newcommand{\defcat}[2]{\DeclareMathOperator{#1}{\mathbf{#2}}}

\defcat{\Set}{Set}

\defcat{\MonFib}{MonFib}
\defcat{\MonFIB}{MonFIB}
\defcat{\Prof}{Prof}
\defcat{\Ci}{C^{\infty}}
\defcat{\linf}{\ell^{\infty}}
\defcat{\Lcvx}{Lcvx}
\defcat{\Pol}{Pol}
\defcat{\Con}{Con}
\defcat{\PROF}{PROF}
\defcat{\The}{Th}
\defcat{\Mod}{Mod}
\defcat{\Comod}{Comod}
\defcat{\SET}{SET}
\defcat{\TOPOS}{TOPOS}
\defcat{\Sketch}{Sketch}
\defcat{\MonSketch}{MonSketch}
\defcat{\MonCat}{MonCat}
\defcat{\MonCAT}{MonCAT}
\defcat{\Adj}{Adj}
\defcat{\SKETCH}{SKETCH}
\defcat{\SkFib}{SkFib}
\defcat{\SkOpFib}{SkOpFib}
\defcat{\SkBiFib}{SkBiFib}
\defcat{\ThFib}{ThFib}
\defcat{\OpFib}{OpFib}
\defcat{\BiFib}{BiFib}
\defcat{\OpFIB}{OpFIB}
\defcat{\Cat}{Cat}
\defcat{\ICAT}{ICAT}
\defcat{\DCAT}{DCAT}
\defcat{\Cart}{Cart}
\defcat{\Site}{Site}
\defcat{\FPCat}{FPCat}
\defcat{\CAT}{CAT}
\defcat{\FIB}{FIB}
\defcat{\Fib}{Fib}
\defcat{\Stack}{Stack}
\defcat{\STACK}{STACK}
\defcat{\Vect}{Vect}
\defcat{\Borel}{Borel}
\defcat{\Lin}{Lin}
\defcat{\FinSet}{FinSet}
\defcat{\FreeLin}{FreeLin}
\defcat{\FreeVect}{FreeVect}
\defcat{\FinVect}{FinVect}
\defcat{\Ban}{Ban}
\defcat{\SepBan}{SepBan}
\defcat{\Space}{Space}
\defcat{\Meas}{Meas}
\defcat{\Sh}{Sh}
\defcat{\Psh}{Psh}
\defcat{\Ball}{Ball}
\defcat{\FreeConvex}{FreeConvex}
\defcat{\talg}{\mc T-Alg}
\defcat{\salg}{\mc S-Alg}

\newcommand{\aln}{\aleph_0}

\newbox\gnBoxA
\newdimen\gnCornerHgt
\setbox\gnBoxA=\hbox{$\ulcorner$}
\global\gnCornerHgt=\ht\gnBoxA
\newdimen\gnArgHgt
\def\uc#1{%
\setbox\gnBoxA=\hbox{$#1$}%
\gnArgHgt=\ht\gnBoxA%
\ifnum     \gnArgHgt<\gnCornerHgt \gnArgHgt=0pt%
\else \advance \gnArgHgt by -\gnCornerHgt%
\fi \raise\gnArgHgt\hbox{$\ulcorner$} \box\gnBoxA %
\raise\gnArgHgt\hbox{$\urcorner$}}

\let \phi\varphi

\newcommand{\rra}[1]{\xrightarrow{\;\;#1\;\;}}

\ifdefined\newextarrow
\newextarrow{\xrra}{{20}{20}{20}{20}}
   {\bigRelbar\bigRelbar{\bigtwoarrowsleft\rightarrow\rightarrow}}
\fi

\newcommand{\sar}[3][]{\ar[phantom,sca={0.4}{start#1}{end#1}]{#2}\ar[Rightarrow,from=start#1,to=end#1,"#3"]}

\let\rla\leftrightarrows

\let\re\relax

\begin{document}
\title{Algebraic theories and commutativity in a sheaf topos}
\thanks{This material is based upon work supported by the National Science Foundation under Award No. DMS-1602279 and Award No. DMS-1246999}
\author{Boaz Haberman}
\maketitle

\begin{abstract}
For any site of definition $\mc C$ of a Grothendieck topos $\mc E$, we define a notion of a $\mc C$-ary Lawvere theory $\tau: \ms C \to \ms T$ whose category of models is a stack over $\mc E$. Our definitions coincide with Lawvere's finitary theories when $\mc C=\aln$ and $\mc E = \Set$.  We construct a fibered category $\Mod^{\ms T}$ of models as a stack over $\mc E$ and prove that it is $\mc E$-complete and $\mc E$-cocomplete. We show that there is a free-forget adjunction $F \dashv U: \Mod^{\ms T} \rla \ms E$. If $\tau$ is a commutative theory in a certain sense, then we obtain a ``locally monoidal closed'' structure on the category of models, which enhances the free-forget adjunction to an adjunction of symmetric monoidal $\mc E$-categories. Our results give a general recipe for constructing a monoidal $\mc E$-cosmos in which one can do enriched $\mc E$-category theory. As an application, we describe a convenient category of linear spaces generated by the theory of Lebesgue integration.
\end{abstract}

\section{Introduction}
\subsection{Lawvere's $\aln$-ary theories}
In this paper, we develop the notion of an algebraic theory internal to a topos of sheaves, and we prove some basic properties. To motivate the definitions, recall that in $\Set$, a finitary Lawvere theory consists of an additive identity-on-objects functor
\[\tau: \aln \to \mc T,\]
where $\aln$ is the category of finite sets and $\mc T$ is a category with finite sums whose objects coincide with the objects of $\aln$. A model of the theory is a presheaf on $\mc T$ which preserves finite products. If $[\mc A, \mc B]_\tm$ denotes the category of all functors from $\mc A$ to $\mc B$ which preserve finite products, then we have
\[\Mod^{\mc T} = [\mc T^{op},\Set]_\tm \into [\mc T^{op},\Set].\]
An {\em algebraic category} is one which is equivalent to $\Mod^{\mc T}$ for some $\mc T$. 

To inform what follows, we recall some well-known facts about algebraic categories. First, every theory $\tau: \aln \to \mc T$ determines a free-forget adjunction
\[F \dashv U: \Mod^{\mc T} \rla \Set.\]
The right adjoint $U$ is conservative, which means that an object of $\Mod^{\mc T}$ is essentially a set with some extra algebraic structure. The left adjoint sends a set $X$ to the free model $FX$ generated by $X$.

Now, the theory $\tau$ can be recovered from the adjunction $F\dashv U$ using the isomorphism
\[\tau n \cong F n.\]
Lawvere's functorial semantics of algebraic theories can thus be summarized as follows: an algebraic theory is identified with the category of finitely-generated free models, and a model of the theory is a product-preserving presheaf on the category of free models.

Next, for any algebraic theory $\tau$, the inclusion $\Mod^{\mc T} \to [\mc T^{op},\Set]$ admits a left adjoint. In particular, the algebraic category $\Mod^{\mc T}$ is complete and cocomplete. 

Finally, we consider commutative theories in which the operations commute which each other. A theory is commutative if and only if the category $\mc T$ admits a tensor product such that the functor $\tau: \aln \to \mc T$ is strict monoidal, i.e. 
\[\tau 1 = I \qq \tau(m \times n) = \tau m \ot \tau n.\]
for $m,n \in \aln$ (for example, for a commutative ring $R$, the theory of $R$-modules has this property). It is well-known that the category of models of a commutative theory admits a {\em closed} symmetric monoidal structure with coherent isomorphisms
\[F1 \cong I \qq F(X \times Y) \cong F X \ot F Y\]
for $X,Y \in \Set$. In other words, the free-forget adjunction $F \dashv U$ is monoidal with respect to the cartesian product in $\Set$.

The purpose of this paper is to generalize all of these facts about $\aln$-ary theories with models in $\Set$ to the setting of $\mc C$-ary theories with models in $\mc E$, where the category of arities $\mc C$ is a standard site and the category of spaces $\mc E$ is the topos of sheaves on $\mc C$.

\subsection{$\mc C$-ary Lawvere theories}
The 2-category of ordinary categories embeds into the 2-category of fibered $\aln$-categories. We can define finite products and sums in fibered $\aln$-categories in purely fibrational terms.

Let $\mc C$ be any category with finite limits. We can view the objects of such a category as potential arities for algebraic operations. To define Lawvere semantics for $\mc C$-ary theories, we simply translate the definitions above into the language of fibered $\aln$-categories and systematically replace $\aln$ with $\mc C$.  

Recall that $\mc C$ itself determines a fibered $\mc C$-category, denoted by $\ms C$, whose fibers are the slice categories
\[\ms C^I := \mc C_{/I}.\]
A $\mc C$-ary theory will then be an identity-on-objects $\mc C$-functor
\[\tau: \ms C \to \ms T\]
that preserves $\mc C$-ary sums. To define the category of models for such a theory, we need to replace the embedding $j: \aln \into \Set$ with a $\mc C$-additive functor $j: \ms C \to \ms E$. If $\mc C$ is equipped with a subcanonical Grothendieck topology $J$, then a natural candidate for $\ms E$ is the stack of toposes whose fibers are defined by
\[\ms E^I = \Sh(\mc C, J)_{/yI},\]
where $y: \mc C \to \Sh(\mc C, J)$ is the Yoneda embedding. Note that $y$ extends to a $\mc C$-additive functor that plays the role of $j$ in this story.

We can now define the category of Lawvere models for $\ms T$ in $\ms E$ by
\[\Mod^{\ms T}_{\Lawvere} = [\ms T^{op}, \ms E]_{\tm} \into [\ms T^{op}, \ms E].\]
In words, a cartesian $\mc C$-functor in $[\ms T^{op}, \ms E]$ is a model of the theory $\ms T$ if it preserves $\mc C$-ary products.

\subsection{$\mc C$-ary Linton theories}
In the preceding paragraph, we have outlined a completely fibrational approach to universal algebra in $\mc E$, which is defined in terms of sums and products indexed by objects of $\mc C$. We now outline an alternative set-theoretic way to define the category of models for a $\mc C$-ary theory, due to Linton.

Let $\tau: \mc C \to \mc T$ be an idenity-on-objects functor, and consider the restriction functor
\[\tau^*: [\mc T^{op},\Set] \to [\mc C^{op},\Set].\] 
Given a Grothendieck topology on $\mc C$, we define the category of Linton models of $\mc T$ in $\Sh(\mc C)$ by
\[\Mod^{\mc T}_{\Linton} = (\tau^*)^{-1} \Sh(\mc C) \into [\mc T^{op},\Set].\]
In words, a presheaf on $\mc T$ is a Linton model if its restriction along $\tau$ is a sheaf on $\mc C$. 

In this setup, the functor $\tau$ need not be additive in any sense, but it is natural to require that the representable presheaves on $\mc T$ be models of the theory. In this case we say that $\tau$ is $\mc C$-continuous.

\subsection{Intrinsic and extrinsic semantics}
To illustrate the two notions of algebraic theory described above, we return to the example of $\aln$. Consider the category $\aln$ of finite sets as a site whose covering families are all finite sum diagrams. An identity-on-objects functor $\tau: \aln \to \mc T$ is finitely additive precisely when it is a Linton theory (the representables on $\mc T$ restrict to sheaves on $\aln$). Similarly, a presheaf $M: \mc T^{op} \to \Set$ restricts to a sheaf on $\aln$ precisely when it preserves finite products. In this case Lawvere semantics are equivalent to Linton semantics.

In general, the $\mc C$-additivity condition and the $\mc C$-continuity condition are completely independent. In particular, $\mc C$-additivity of $\tau$ is an intrinsic category-theoretic notion, while $\mc C$-continuity (of the ordinary functor $\tau^1: \ms C^1 \to \ms T^1$, say) is a set-theoretic notion which, moreover, depends on the chosen topology for $\mc C \cong \ms C^1$. For the same reason, the category $\Mod^{\ms T^1}_{\Linton}$ of Linton models in $\ms E^1$ will generally be different from the category $\Mod^{\ms T}_{\Lawvere}(\ms E)$ of Lawvere models in $\ms E$.

We will show, however, that when $\tau$ is $\mc C$-additive and each $\tau^I$ is $\ms C^I$-continuous, then the two algebraic categories we have defined are equivalent. Thus, we can freely pass between Lawvere semantics (defined in terms of fibered categories) and Linton semantics (defined in terms of sheaves). 

To state this equivalence precisely we must sharpen the definitions a bit. Let $\ms T$ be a category with $\mc C$-ary sums. Lawvere semantics defines a category of models of $\ms T$ in any category with $\mc C$-ary products. We have
\[\Mod^{\ms T}_{\mathrm{Lawvere}}(\ms D) : = [\ms T^{op}, \ms D]_\tm.\]
That is, a model of the theory $\ms T$ in the fibered $\mc C$-category $\ms D$ is a fibered $\mc C$-functor from $\ms T^{op}$ to $\ms D$ which preserves fibered $\mc C$-products. Note that this definition does not require $\ms D$ to be the category of sheaves on $\mc C$.

Linton's semantics is defined in terms of sheaves. Suppose that $\tau: \ms C \to \ms T$ is a fibered functor, and that each  component $\tau^I$ is $\ms C^I$-continuous (with respect to the induced topology on $\ms C^I$). A model of $\ms T^I$ (in $\Sh(\ms C^I) \cong \ms E^I$) is a presheaf on $\ms T^I$ whose restriction to $\ms C^I$ is a sheaf. That is, we have
\[(\Mod^{\ms T^I}_{\mathrm{Linton}}) : = ((\tau^I)^*)^{-1} \Sh(\ms C^I).\]

We are now in a position to compare the two categories of models and state our main technical theorem.

\begin{definition}
Let $\mc C$ be a standard Grothendieck site, and let $\ms C$ be the codomain fibration. A standard $\mc C$-ary theory is a $\mc C$-additive identity-on-objects fibered $\mc C$-functor
\[\tau: \ms C \to \ms T\]
such that every representable presheaf on $\ms T^I$ restricts along $\tau^I$ to a sheaf on $\ms C^I$.
\end{definition}

The main technical theorem states for standard $\mc C$-ary theories, the two notions of model are equivalent.
\begin{theorem}
Let $\tau$ be a standard $\mc C$-ary theory. There is an equivalence of fibered $\mc C$-categories
\[\Mod^{\ms T^I}_{\mathrm{Linton}} \cong \Mod_{\mathrm{Lawvere}}^{\ms T}(\ms E^I),\]
where $\ms E$ is the canonical fibering of $\Sh(\mc C)$ over $\mc C$.
\end{theorem}

This equivalence turns out to be very useful, because in the Linton theory setting we can exploit set-theoretic tools such as the theory of sketches and locally presentable categories. In the Lawvere theory setting we can exploit category-theoretic tools; in particular, we use the fact that $\ms E$ is a stack over $\mc C$.

\subsection{Completeness and cocompleteness}
It is not hard to show that the $\mc C$-category $\Mod^{\ms T}$ of models is $\mc C$-complete and $\mc C$-cocomplete. In the case $\mc C = \aln$ this statement means that the category of models has finitary limits and colimits.

The category of models for an $\aln$-ary theory also has $\Set$-ary limits and colimits. Since a $\mc C$-ary theory has models in the topos $\mc E$, it is natural to ask whether the category of models has $\mc E$-ary limits and colimits.  

In order for this question to make sense, we need to extend $\Mod^{\ms T}$ to a fibered $\mc E$-category. That this is possible is our next result:
\begin{theorem}
The $\mc C$-category $\Mod^{\ms T}$ is a stack over the site $(\mc C, J)$.
\end{theorem}
In particular, by the comparison lemma, this implies that there is an essentially unique extension of $\Mod^{\ms T}$ to a fibered $\mc E$-category, such that the resulting extension is a stack for the canonical topology on $\mc E$.

By abuse of notation, we denote the extension by $\Mod^{\ms T}$. In this extension, we have an $\mc E$-ary analogue of completeness and cocompleteness:

\begin{theorem}
The $\mc E$-category $\Mod^{\ms T}$ is fiberwise $\Set$-complete and cocomplete, and has $\mc E$-ary sums and products.
\end{theorem}

The existence of $\mc E$-ary sums is closely related to the existence of free models. Recall that an $\aln$-ary theory $\tau: \aln \to \mc T$ is the restriction to $\aln$ of the free model functor $F: \Set \to \Mod^{\mc T}$. Analogously, a $\mc C$-ary theory $\tau: \ms C \to \ms T$ is the restriction to $\ms C$ of the free model $\mc E$-functor $F: \ms E \to \Mod^{\ms T}$. In particular, this functor exists:

\begin{theorem}
The forgetful $\mc E$-functor $U: \Mod^{\ms T} \to \ms E$ has a fibered $\mc E$-adjoint $F: \ms E \to \Mod^{\ms T}$, whose restriction to $\ms C$ is isomorphic to $\tau$.
\end{theorem}

To construct $\mc E$-ary sums and products, we follow an idea of Linton. We consider the free model functor $F: \ms E \to \Mod^{\ms T}$ to be an $\mc E$-ary theory 
\[\bar \tau: \ms E \to \bar{\ms T}\]
where $\bar{\ms T}$ is the full image of $F$ and $\bar \tau$ is the corestriction of $F$ to $\bar{\ms T}$. There is an equivalence of categories
\[\Mod^{\ms T} \cong \Mod^{\bar{\ms T}}.\]
Applying our general machinery to the large theory $\bar{\tau}$ gives existence of $\mc E$-ary products and sums in $\Mod^{\bar{\ms T}}$, and hence in $\Mod^{\ms T}$ by equivalence. While this argument is conceptually simple, it contains some technical difficulties, because $\mc E$ is a large category.

\subsection{Tensor products}
Our next result concerns the construction of tensor products of models. By a tensor product, we mean a closed symmetric $\mc E$-monoidal structure on $\Mod^{\ms T}$, such that the free model $\mc E$-functor is strong $\mc E$-monoidal. This means that there are coherent isomorphisms
\[F^P(X\times_P Y) \cong F^P X \ot_P F^P Y \qq\qq F^P 1_P \cong I_P\]
for each $P \in \mc E$ and $X,Y \in \ms E^P$.

Recall that the arity category $\mc C$ is assumed to have finite limits. It follows that if $F$ is strong $\mc E$-monoidal, then the image of $\ms C$ under $F$ is closed under tensor products. Thus, if $\Mod^{\ms T}$ admits tensor products, then the original theory $\tau: \ms C \to \ms T$ can be made into a strong $\mc C$-monoidal functor.

Conversely, there is at most one symmetric $\mc C$-monoidal structure on the theory $\ms T$ such that the functor $\tau$ is strict monoidal. When such a structure exists, we say that the theory is {\em commutative}. For every commutative theory, the category of models admits a essentially unique tensor product.

We can now state our main result on commutative theories:
\begin{theorem}
If $\tau: \ms C \to \ms T$ is a commutative then there is a closed symmetric monoidal $\mc E$-category structure on $\Mod^{\ms T}$, such that the free model functor $F: \ms E \to  \Mod^{\ms T}$ is strong $\mc E$-monoidal.
\end{theorem}

It is well-known that there is a  monoidal structure on $[\mc T^{op}, \Set]$, namely the Day convolution. Moreover, Day's reflection theorem gives a condition under which the convolution on $[\mc T^{op},\Set]$ determines a closed monoidal structure on $\Mod^{\mc T}$. The import of our result is that this condition holds, and that the tensor product is compatible with the $\mc E$-category structure and the free-forget adjunction.

In the terminology of~\cite{Shulman2013}, the category of models for a commutative theory is an $\mc E$-cosmos. In particular, it is a suitable setting for formal enriched category theory.

\subsection{Examples}
Our notion of universal algebra necessarily depends on the site of definition $\mc C$. For example, the topos of sets is generated by the unit site, i.e.
\[\Set \cong \Psh(\mathbf 1).\]
A $\mathbf 1$-ary theory $\tau: \mathbf 1 \to \mc T$ is a monoid, and a model of $\mc T$ is a set with a monoid action. Clearly, not every finitary theory is a $\mathbf 1$-ary theory. To arrive at Lawvere's original definition, we present $\Set$ as
\[\Set \cong \Sh(\aln),\]
where the topology on $\aln$ is generated by finite sum diagrams. In the opposite direction, we can enlarge the site $\aln$ by allowing sets with larger cardinality to arrive at a notion of $\ld$-ary theory for some regular cardinal $\ld$.

In this paper, we are not at all interested in transferring the notion of finiteness to an arbitrary topos. We are mainly concerned with sites $\mc C$ whose objects are continua (with uncountably many points). Our motivating example is the theory of Lebesgue integration; this theory has an $\R$-ary operation
\[\int \,d\mu: V^\R \to V\]
for every finite Borel measure $\mu$ on $\R$.

\subsection{A convenient category of linear spaces}
In the conclusion of the paper, we flesh out the last example above by taking $\mc C = \Borel$ where $\Borel$ is the category of Borel maps whose objects are
\[\Ob(\Borel) = \{0,1,2,\dotsc, \N, \R\}.\]
This is a site with respect to the finite sum topology. A sheaf on $\Borel$ is a {\em nonlinear space}. If $X$ is a nonlinear space, we interpret the set $X^{I}$ as the collection of {\em bounded} Borel maps from $I$ to $X$. 

Next, we define a commutative $\Borel$-ary theory whose $I$-ary operations are finite measures on $I$. A model of the theory is a {\em linear space}. That is, if $V$ is a linear space, then for each $I$-ary operation $\mu: 1 \to FI$, there is an integration map
\[\int_{i\in I} \,d\mu(i): V^I \to V^1.\]
sending $f: I \to V$ to $\int f(i) \,d\mu(i)$. A morphism of models is a morphism of the underlying spaces which commutes with the Lebesgue integral, i.e. a bounded linear map.

The category of linear spaces contains, as a full subcategory, the category of separable Banach spaces and bounded linear maps. Moreover, it contains the interesting spaces used in applied functional analysis. In particular, it contains spaces of smooth functions, spaces of Schwartz functions, spaces of test functions, and corresponding dual spaces.

For each nonlinear space $P$, there is a category $\ms E^P$ of nonlinear bundles over $P$ and a category $\ms V^P$ of linear bundles over $P$. If $V = \bk{V_p}_{p\in P}$ and $W=\bk{W_p}_{p\in P}$ are objects of $\ms V^P$, then a bundle map $T: V\to W$ is a collection of linear maps $T_p: V_p \to W_p$, one for each $p$, and the assignment $p \mapsto T_p$ is measurable and bounded with respect to the parameter $p$.

If $\phi: P \to Q$ is a bounded measurable map, then there is a pullback functor $\phi^*: \ms V^Q \to \ms V^P$ with left and right adjoints $\phi_!$ and $\phi_*$. Intuitively, these are given by the formulas
\begin{align*}
(\phi^* V)_p &= V_{\phi(p)}\\
(\Si_{\phi} V)_q & = \sum_{p\in \phi^{-1}(q)} V_p\\
(\Pi_\phi V)_q & = \prod_{p \in \phi^{-1}(q)} V_q.
\end{align*}
Here, a section $\mu: P \to \sum_{p\in P} V_p$ is (morally) a finite measure on $P$ such that $d\mu(p)$ lies in $V_p$ for each $p$. A section $f: P \to \prod_{p\in P} V_p$ is a bounded map on $P$ such that $f(p) \in V_p$ for each $p$.

For every bundle $X = \bk{X_p}_{p\in P}$ of nonlinear spaces, there is a bundle of free linear spaces $F^P X$, given by
\[\bk{F^P X}_p = F X_p,\]
where $FX$ is the space of finite measures on the nonlinear space $X$. In particular, we have a well-defined notion of measure on any nonlinear space $X$.

Since the theory is commutative, we also have a (projective) tensor products of linear bundles, with
\[(V \ot_P W)_p = V_p \ot W_p.\]
If $v: I \to V$ and $w: J \to W$ are bounded families of vectors, and $\mu :1\to F(I\times J)$ is a finite measure, then 
\[\int_{I \times J} v(i) \ot w(j) \,d\mu(i,j)\]
is an element of $V \ot W$.

The tensor product is closed, which means that it determines an internal hom of linear bundles. The fibers
\[[V,W]_p = [V_p, W_p],\]
are spaces of bounded linear maps.

A commmon pattern in the harmonic analysis and in the study of partial differential equations is to introduce various parameters into a problem, and to establish multilinear inequalities which are uniform in these various parameters. For example, the entire field of semiclassical analysis can be caricatured as ``harmonic analysis with respect to a small parameter $h$''. We believe that our framework of linear bundles over a parameter space provides a good language to formulate and keep track of estimates which arise in this context.

\section{Related work}
\subsection{Varieties of algebra}
Birkhoff~\cite{Birkhoff1935} defined a variety of algebras to be a collection of finitary operations related by universally quantified equational axioms. He then characterized those categories whose arrows are homomorphisms of algebras for some variety $\tau$. Lawvere's thesis~\cite{Lawvere1963} showed that the data for a variety of algebras can be canonically represented as the category of finitely-generated free models. More importantly, he showed that a homomorphism of algebras is the same thing as a natural transformation between multiplicative presheaves on this category. This setup was generalized by Linton in~\cite{Linton1966} to encompass theories with infinitary operations. In~\cite{Linton1969}, Linton formulated a theory where the collection of arities is an arbitrary category instead of a subcategory of $\Set$.

The theory of sketches was developed by Ehresmann~\cite{Ehresmann1968}, Kennison~\cite{Kennison1968} and Gabriel-Ulmer~\cite{GabrielUlmer1971}. In this theory, the finite sums in Lawvere's semantics are replaced by an arbitrary colimit cones. A sketch can be interpreted as a multi-sorted theory which allows, moreover, some partially defined operations. 

\subsection{Enriched and internal theories}
The notion of a {\em finitary} algebraic theory can be generalized to categories other than $\Set$ in various different ways. Borceux and Day~\cite{BorceuxDay1980} define a notion of a finitary theory in a certain type of closed category, and the arities of this theory are finite multiples of the unit object. This uses a $\Set$-based notion of finiteness. 

There are also intrinsic notions of finiteness generalizing the $\Set$-based notion. Johnstone and Wraith~\cite{JohnstoneWraith1978} define a notion of finitary theory internal to an elementary topos, where the notion of finiteness is connected to the natural numbers object. Kelly in~\cite{Kelly1982a} defines a notion of a finite {\em limit} theory in the enriched setting, the limits here are finitely presentable in a certain sense. Similarly, Power~\cite{Power1999} defines a notion of a finite {\em product} theory, whose models are functors preserving {\em powers} indexed by finitely presentable objects.

A very general notion of enriched sketch is given in Kelly's monograph~\cite{Kelly1982b}. One can associate an enriched sketch to a category in a canonical way by choosing all limits with some predetermined indexing type. For example, this could include finitely-presentable powers, but also $\k$-presentable powers for some regular cardinal $\k$. In Lack and Rosick\'y's paper~\cite{LackRosicky2011} some examples are given of {\em sound limit doctrines} in enriched categories, where explicit constructions are available. The notion of Lawvere theory is given relative to such a doctrine. They show that in this setting many of the associated constructions can be made more explicit. However, the notion of a sound limit doctrine is rather inflexible, and it is shown in~\cite{AdamekBorceuxLackRosicky2002} that for $\ld > \aln$, the doctrine of $\ld$-ary products in $\Set$ is not sound in this sense. Lucyshyn-Wright~\cite{Lucyshyn-Wright2016} defines a general notion of enriched Lawvere theory which allows the category of arities to be more or less arbitrary and works in a (not necessarily cartesian) closed category; the main difference between his setup and ours is that ours is much less general, but also significantly more explicit.

Our approach eschews the use of enriched category theory in favor of the theory of fibered ordinary categories, which we have found to be more straightforward to work with. We also restrict our attention to Grothendieck toposes over $\Set$; this ensures that the models for our theories are reflective subcategories of ordinary presheaf categories. The main conceptual difference between our approach and other approaches to universal algebra is that instead of considering a single theory $\tau: \mc C \to \mc T$, we consider a whole family of theories $\tau_I: \ms C_I \to \ms T_I$. This ensures that the notion of model is stable under localization and gives us a well behaved fibered $\mc E$-category of models. Of course, this is not a new idea at all, and is closely modeled after~\cite[SGA 4.IV]{ArtinGrothendieckVerdier1972}. Most of our constructions are well-known and use existing techniques. However, it is difficult to track down all of the necessary results in the literature. As we would like to use the results of this paper in future work, we have found it necessary to write them down.

\subsection{Convenient vector spaces}
The main inspiration for this paper was the work of Fr\"olicher and Kriegl~\cite{FrolicherKriegl1988} on convenient spaces. They defined a convenient category of topological vector spaces with a closed monoidal structure. These spaces are associated with a cartesian closed category $\Ci$ of smooth spaces and a monoidal adjunction $F\dashv U: \Con \rla \Ci$. A smooth space is a set together with a collection of smooth curves satisfying some axioms. This is close in spirit to the idea of a sheaf, but the setness and the extra axioms prevent $\Ci$ from being locally cartesian closed. They also exhibited a monoidal adjunction $\ell^1 \dashv U: \Con \rla \linf$, where $\linf$ is a locally cartesian closed category of bounded spaces. These spaces are equipped with a collection of bounded sequences satisfying some axioms. The space $\ell^1\R$ consists of countably supported measures on $\R$ with bounded mass. Naively, one might expect to be able to replace $\ell$ with $L$ and obtain a notion of a space with bounded measurable curves. However $\Con$ includes the category of Banach spaces as a full subcategory, and Banach spaces already do not have a good notion of Lebesgue integration in general. Moreover, the whole setup is somewhat baroque and apparently miraculous due to the mixture of algebraic and topological definitions.

\subsection{The Giry-Lawvere monad}
Lawvere proposed using measure theory as a way to encode algebraic structures on topological spaces, and this was worked out in Giry's paper~\cite{Giry1982} in the case of Polish spaces. He defines a monad $M$ on the category $\Pol$ of Polish spaces and {\em continuous} maps, sending $X$ to the space $MX$ of probability measures on $X$, whose weak topology makes it a Polish space. Thus one can define a convex Polish space to be an algebra for this monad. However, this entails the rather severe restriction that the underlying space is Polish. It also seems more appropriate to consider a category of measurable spaces and {\em measurable} maps.

One might hope to define the Giry-Lawvere monad in the context of measurable spaces and measurable maps. This is not so convenient, however, because the category of measurable spaces and measurable maps is not cartesian closed~\cite{Aumann1961}. For this reason Heunen, Kammar, Staton and Yang~\cite{HeunenKammarStatonYang2017} defined a notion of a quasi-Borel space in terms of concrete sheaves. The category of quasi-Borel spaces admits a monad $M$ sending a quasi-Borel space to a space of probability measures. Their definition of a convex space is thus very similar in spirit to our definition of a linear space, but we have found it more convenient to work with arbitrary sheaves instead of concrete sheaves.

\section{Preliminaries}
\subsection{Size constraints}
We follow standard conventions for categories of sets and categories. An object of $\Set$ is a small set. A  category is in $\Cat$ if the set of arrows is small. A category is in $\CAT$ if it is locally small and the set of arrows is moderate. A category is in $\CAT'$ if the set of arrows is moderate.

\subsection{Presheaves and fibered categories}
We write $[\mc C,\mc D]$ for the category of functors from $\mc C$ to $\mc D$. We write
\[\hat{\mc C} = \Psh^{\mc C} = [\mc C^{op},\Set]\]
for the category of presheaves on $\mc C$. 

By a $\mc C$-category, we mean a functor $p: \mc F \to \mc C$. If $p$ is a {\em fibered} $\mc C$-category, we may identify it with a pseudofunctor $\ms F: \mc C^{op} \to \CAT$. We write $\ms F \in \FIB^{\mc C}$ if $\ms F$ is a fibered $\mc C$-category, and we write $\ms F \in \Fib^{\mc C}$ if each fiber $\ms F^I$ is small. 

\subsection{Restriction and extension}
If $\ms F$ is a fibered $\mc C$-category and $\phi: I \to J$, we denote restriction along $\phi$ by 
\[\ms F^\phi: \ms F^J \to \ms F^I.\]
Dually, if $\ms F$ is an opfibered $\mc C$-category (i.e., $p^{op}: \mc F^{op} \to \mc C^{op}$ is a fibered $\mc C^{op}$-category), we denote {\em left extension} along $\phi$ by
\[\ms F_\phi: \ms F_I \to \ms F_J.\]
If $\ms F$ if both fibered and opfibered over $\mc C$, we say it is a bifibered $\mc C$-category. In this case we have an adjunction $\ms F_\phi \dashv \ms F^\phi$. 

Note that, for an {\em object} $I \in \mc C$, we have $\ms F^I = \ms F_I$ by definition. If $\ms F$ is a fibered $\mc C$-category, its dual $\ms F^{op}$ is defined (as a pseudofunctor) by $(\ms F^{op})^I = (\ms F^I)^{op}$. If  $\ms F^{op}$ is bifibered, then there are {\em right extension} functors $\ms F_{\phi_*}$ such that $\ms F^\phi \dashv \ms F_{\phi_*}$.

We will use the standard notations
\[\phi_! = \ms F_\phi, \qq \phi^* = \ms F^\phi, \qq \phi_* = \ms F_{\phi*}\]
when the fibered $\mc C$-category $\ms F$ is clear from the context.

As a special case, we will write
\[\Psh_F \dashv \Psh^F \dashv \Psh_{F_*}.\]
Thus $\Psh^F$ is precomposition with $F$, and the adjoints $\Psh_F$ and $\Psh_{F_*}$ correspond to left and right Kan extension along $F$.

\subsection{$\mc C$-sums and $\mc C$-products}
Let $\mc C$ is a category with finite limits, and let $\ms F$ be a bifibered $\mc C$-category. Given a pullback diagram in $\mc C$ 
\[\begin{tikzcd}
\cd  \ar[swap]{d}{\ti X}\ar{r}{\ti Y}& \ar{d}{X}\\
 \ar[swap]{r}{Y} & .
\end{tikzcd}\]
Then one can define a canonical transformation
\[\ti Y_! \ti X^* \to X^* Y_!.\]
If this transformation is an isomorphism, we say that the Beck-Chevalley condition holds for the pullback diagram.  

We say that $\ms F$ has $\mc C$-ary sums if it is bifibered and the Beck-Chevalley condition holds for every pullback in $\mc C$. Similarly, we say that $\ms F$ has $\mc C$-ary products if $\ms F^{op}$ has $\mc C$-ary sums.

\subsection{Change of base}
It is straightforward to define restriction 2-functors
\[F^*: \Fib^{\mc D} \to \Fib^{\mc C}\]
which make $\Fib$ into a fibered $\Cat$-category. The restriction $F^*$ has left and right 2-adjoints $F_!$ and $F_*$ (constructed in~\cite[I.2.4]{Giraud1971}, where the left 2-adjoint $F_!$, for example, is characterized by {\em equivalences}
\[\Fib^{\mc D}(F_! \ms F, \ms G) \cong \Fib^{\mc C}(\ms F, F^* \ms G).\]
Note that these are not isomorphisms, and $\Fib$ does not have left and right extensions in the sense we have earlier described. 

\section{Lawvere theories}
\subsection{Additive and multiplicative $\mc C$-functors}
Let $\mc C$ be a fixed category with finite limits. Fix $\ms G$ and $\ms H$ with $\mc C$-ary sums. A fibered $\mc C$-functor $F: \ms G \to \ms H$ is $\mc C$-additive if, for any $\phi \in \mc C$, the canonical transformation $\phi_! F \to F \phi_!$ is an isomorphism. 

Dually, if $\ms F$ and $\ms G$ have $\mc C$-ary products, we say $F$ is $\mc C$-multiplicative if, for every $\phi\in \mc C$ the canonical transformation $F\phi_* \to \phi_* F$ is an isomorphism.

We write $[\ms G,\ms H]$, $[\ms G,\ms H]_+$ and $[\ms G,\ms H]_\times$ for the fibered $\mc C$-categories of $\mc C$-fibered, $\mc C$-additive, and $\mc C$-multiplicative functors, respectively. Thus the fibers of $[\ms G,\ms H]$ are given by
\[[\ms G,\ms H]^I := \FIB^{\mc C}(I \times \ms G,\ms H). \]
The subcategories $[\ms G,\ms H]_+$ and $[\ms G,\ms H]_\tm$ are fibered as well, due to the Beck-Chevalley condition (for example, given $\phi \in \mc C$, the fibered restriction $[\phi, \ms H]$ is $\mc C$-additive).

\subsection{$\mc C$-ary operations}
Let $\mc C$ be a small category with finite limits, and let $\ms A$ be any category with $\mc C$-ary products. A $\phi$-ary operation in the category $\ms A$ is an arrow
\[f: \phi_* A \to B.\]
When $\mc C=\Set$ and $\ms A^I = \mc A^I$ for some ordinary category $\mc A$, the object $A$ is an $I$-indexed family $\bk{A_i}_{i\in I}\subset \mc A$, and the object $B$ is a $J$-indexed family $\bk{B_j}_{j\in J}\subset \mc A$. The arrow $f$ consists of, for each $j \in J$, an arrow
\[f_j: \prod_{\phi(i)=j} A_i \to B_j.\]
Thus, for each $j$, the arrow $f_j$ is a $\phi^{-1}(j)$-ary operation in the category $\mc A$.

\subsection{$\mc C$-ary theories}
Following Lawvere, we may view any category $\ms T$ with $\mc C$-sums as a multi-sorted algebraic theory. If $\ms E$ is a category with $\mc C$-products, we define a Lawvere model of $\ms T$ in $\ms E$ to be an object of 
\[\Mod^{\ms T}_\Lawvere(\ms E)=[\ms T^{op}, \ms E]_\tm.\]
An object $T\in \ms T^I$ represents an abstract $I$-ary family of sorts. An arrow of type $\al: \phi_* S \to T$ in $(\ms T^I)^{op}$ represents an abstract $\phi$-ary operation. The axioms of the theory are all encoded in the composition map of $\ms T$.

If $M$ is a Lawvere model of the theory $\ms T$ in some category $\ms E$, then we may interpret each abstract $I$-ary family of sorts $T \in \ms T^I$ as the concrete $I$-ary family 
$M^IT \in \ms E^I$
of objects in $\ms E$. Similarly, we may interpret each abstract $\phi$-ary operation $\al: \phi_* S \to T$ in $(\ms T^I)^{op}$ as the concrete $\phi$-ary operation
\[\phi_* M^JS \cong M^I\phi_* S \rra{M^I\al} M^IT\]
in the category $\ms E$. Note that the interpretation of $\al$ as a $\phi$-ary operation is only possible because the functor $M$ is assumed to be $\mc C$-multiplicative. 

The functoriality of $M$ encodes the satisfaction relation. That is, the concrete $\phi$-ary operations determined by $M$ must satisfy the axioms of the theory $\ms T$.

\subsection{Single-sorted theories}
Given a fixed category $\mc C$ with finite limits, we define the fibered $\mc C$-category $\ms C$ to be the codomain fibration $\cod: \mc C^{\da} \to \mc C$, so that the fiber $\ms C^I$ is the slice category of arrows over $I$, i.e.
\[\ms C^I \cong \mc C_{/I}.\]
The $\mc C$-category $\ms C$ is the free $\mc C$-sum completion of the unit $\mc C$-category. This means that if $\ms D$ has $\mc C$-products, then there is an equivalence of $\mc C$-categories
\[\Mod^{\ms C}_{\Lawvere}(\ms D) = [\ms C^{op}, \ms D]_\tm \cong [1, \ms D] \cong \ms D.\]
Thus a model of $\ms C$ in $\ms D$ is essentially the same as an object of $\ms D$. In this sense, the $\mc C$-category $\ms C$ is the $\mc C$-ary theory of {\em objects}. A single-sorted theory has the same sorts as $\ms C$, but may have additional operations as well.

\begin{definition}
Let $\mc C$ be a small category with finite limits. A $\mc C$-ary Lawvere theory is a $\mc C$-additive identity-on-objects functor
\[\tau: \ms C \to \ms T,\]
where $\ms T$ is a small category with $\mc C$-sums.
\end{definition}
If $\tau$ is a $\mc C$-ary Lawvere theory, then precomposition with $\tau$ is a forgetful fibered $\mc C$-functor
\[U_{\Lawvere}^\tau(\ms D): \Mod^{\ms T}_{\Lawvere}(\ms D) \to \Mod^{\ms C}_{\Lawvere}(\ms D) \cong \ms D.\]
Thus a model of $\ms T$ in $\ms D$ is an object of $\ms D$ with some extra structure and properties. 
\section{Linton theories}
\subsection{Standard sites}
We recall that a {\em standard site} is a small finite limit category $\mc C$ equipped with a subcanonical topology.  This notion is stable under localization: for each object $I \in \mc C$, the slice category $\mc C_{/I}$ has finite limits and an induced subcanonical topology. The topology on $\mc C_{/I}$ so that there is an equivalence of fibered $\mc C$-categories
\[\Sh(\mc C_{/I}) \cong \Sh(\mc C)_{/yI}\]
betweeen sheaves on the site $\mc C_{/I}$ and sheaves on $\mc C$ with maps to $yI$. Here 
\[y:\mc C \to \Sh(\mc C)\]
is the Yoneda embedding.

\subsection{Linton models}
In Linton's semantics, an algebraic $\mc C$-ary theory is an identity-on-objects functor $\tau: \mc C \to \mc T$. We refer to such a functor as a {\em Linton theory}. Given a replete subcategory $\mc E \subset \Psh^{\mc C}$, a {\em Linton model} of $\mc C$ in $\mc E$ is presheaf $M \in \Psh^{\mc T}$ such that the restriction $\tau^* M$ lies in $\mc E$. Note that the category of Linton models is defined in an essentially set-theoretic way.

In this paper we specialize to the case where $\mc E$ is the sheaf topos $\Sh(\mc C)$. Thus a Linton model of $\tau$ is a presheaf on $\mc T$ whose restriction to $\mc C$ is a sheaf. We say that $\tau$ is {\em subcanonical} if, for every $T \in \mc T$, the representable presheaf $y T$ is a Linton model.

\begin{definition}
Let $\mc C$ be a standard site, and let $\tau: \ms C \to \ms T$ be a $\mc C$-ary Lawvere theory. We say that $\tau$ is {\em standard} if, for every $I \in \mc C$ and $T \in \ms T^I$, the Linton theory 
\[\tau^I: \ms C^I \to \ms T^I\]
is subcanonical (with respect to the induced topology on $\ms C^I$).
\end{definition}

Let $\tau:\ms C\to\ms T$ be a standard $\mc C$-ary Lawvere theory.  By the definition, the $\mc C$-category $\ms T$ is a bifibration, and we may define a $\mc C$-fibered category $\Psh^{\ms T}$ by the formula
\[(\Psh^{\ms T})^I = \Psh^{\ms T_I}.\]
That is, the restriction along $\phi$ is defined to be precomposition with the left extension $\ms T_\phi$. 

We may now define the full subfibration 
\[\Mod^{\ms T}_{\Linton} \subset \Psh^{\ms T}.\]
 First, we describe the fibers: a presheaf on $\ms T^I$ is in $(\Mod^{\ms T}_{\Linton})^I$ if it is a Linton model of $\tau^I$; i.e. if its precomposition with $\tau^I$ is a sheaf on $\ms C^I$. Next, we claim that $\Mod^{\ms T}_{\Linton}$ is closed under restriction. That is, if $M\tau^J$ is a sheaf, then $M\ms T_\phi \tau^I$ is a sheaf as well. In fact, by $\mc C$-additivity of $\tau$, we have
\[M\ms T_\phi \tau^I \cong M \tau^J \ms C_\phi,\]
and precomposition with $\ms C_\phi$ preserves sheaves. 

Note that, by definition, a Linton model of the degenerate theory $\id: \ms C\to \ms C$ is the same thing as a sheaf. That is, we have an equivalence of fibered $\mc C$-categories
\[\Mod^{\ms C_I}_{\Linton} \cong \Sh(\mc C_{/I}) \cong \ms E^{I},\]
where $\ms E$ is the canonical fibering of $\Sh(\mc C)$ over $\mc C$. Thus $\ms C$ can be viewed as the theory of objects (in $\ms E$). In particular, we have a forgetful fibered $\mc C$-functor 
\[U^\tau_{\Linton}: \Mod^{\ms T}_{\Linton} \to \ms E.\] 
We claim that, for a standard $\mc C$-ary theory, the functors $U^\tau_{\Linton}$ and $U^\tau_{\Lawvere}(\ms E)$ are essentially the same. In order to compare them, we define a cartesian $\mc C$-functor
\[\al: \Mod^{\ms T}_\Lawvere(\ms E) \to \Mod^{\ms T}_{\Linton}\]
and show that it is a $\mc C$-fibered equivalence. 
\begin{theorem}
For $J \in \mc C$, let $\lp_J: J \to J\times J$ be the diagonal map. If $\tau$ is a standard $\mc C$-ary theory, then the functors
\[\al_J :\Mod^{\ms T}_\Lawvere(\ms E^J) \to (\Mod^{\ms T}_{\Linton})^J\]
defined by
\[\al_J(M)= \ms E^{J \times J}(\lp_J, M^J-)\]
are the components of a cartesian $\mc C$-functor $\al$ fitting into a diagram
\[
\begin{tikzcd}
\Mod^{\ms T}_\Lawvere(\ms E)\car{dr}{\cong}\ar{d} \ar{r}{\al}&\ar{d}(\Mod^{\ms T}_{\Linton})\\
\ms E\ar{r} & \ms E
\end{tikzcd}\]
Moreover, the cartesian functor $\al$ is an equivalence in $\FIB^{\mc C}$.
\end{theorem}
\begin{proof}
For $M: \ms T^{op} \to \ms E^J$ cartesian and $\mc C$-multiplicative, the components of $M$ are functors 
\[M^I: \ms T_I^{op} \to \ms E^{I\times J}.\]
By definition, we have
\[U^J_{\Lawvere} M = M^1 \tau^1 1 \in \ms E^J.\]
Writing $E \in \mc C_{/J}$ as $E \cong \ms C_{E} \ms C^E 1_J$, we compute
\begin{align*}
(U^J_{\Linton} \al_J M) E& = \ms E^{J \times J}(\lp_J, M^J \tau^J E) \\
& \cong \ms E^{J \times J}(\lp_J, M^J \tau^J \ms C_{E} \ms C^E 1_J) \\
& \cong \ms E^{J}(1_J,\ms E^{\lp_J} \ms E_{E\times J*} \ms E^{E\times J} M^J \tau^J 1_J)\\
&\cong \ms E^J(1_J, \ms E_{E*} \ms E^{(\id_J, E)}\ms E^{E \times J} M^J \tau^J 1_J)\\
&\cong \ms E^J(1_J, \ms E_{E*} \ms E^{E}\ms E^{\lp_J} M^J \tau^J 1_J)\\
&\cong \ms E^J(E, \ms E^{\lp_J} M^J \tau^J 1_J)\\
&\cong \ms E^J(E, \ms E^{\lp_J} \ms E^{\pi_J} M^1 \tau^1 1) \\
&\cong \ms E^J(E, M^1 \tau^1 1) \\
&=U^J_{\Lawvere} M E.
\end{align*}
In particular, we see that $\al_J M$ is, in fact, a Linton model (because $U^J_{\Lawvere} M \in \ms E^J$).

Now we need to show that $\al_J$ is pseudonatural in $J$. For $\phi: K \to J$ we claim there is a canonical isomorphism
\[\begin{tikzcd}
\car{dr}{\cong}\re [\ms T^{op}, \ms E^J] _\tm\ar{r}{\al_J} \ar{d}[']{\ms E^\phi \cp}&\Psh^{\ms T_J}\ar{d}{\Psh^{\ms T_{\phi}}}\\
 \re[\ms T^{op},\ms E^K]_\tm \ar{r}[']{\al_K} &\Psh^{\ms T_K}.
\end{tikzcd}
\]
Indeed, for $M: \ms T^{op} \to \ms E^J$ and $T \in \ms T_K$ we have
\begin{align*}
(\Mod^{\ms T_{\phi}} \al_J M)(T)& = (\al_J M)(\ms T_{\phi} T)\\
&=\ms E^{J\times J}(\lp_J, M^J \ms T_{\phi} T)\\
&\cong \ms E^{J \times J}(\lp_J, \ms E_{\phi \times J*} M^K T)\\
&\cong \ms E^{K \times J}(\ms C^{\phi \times J} \lp_J, M^KT)
\end{align*}
On the other hand, we have
\begin{align*}
(\al_K \ms E^\phi M)(T) & = \ms E^{K \times K}(\lp_K, (\ms E^\phi M)^K T)\\
&= \ms E^{K \times K}(\lp_K, \ms E^{\phi \times K} M^K T)\\
&\cong \ms E^{J \times K}(\ms C_{\phi\times K} \lp_K, M^KT).
\end{align*}
But we have a canonical isomorphism $\ms C^{\phi \times J} \lp_J \cong \ms C_{\phi\times K} \lp_K$, in light of the commutative diagram
\[\begin{tikzcd}
K\opcart{dr}\ar{r} \ar{d}{\lp_K}& K \ar{d}\cart{dr}\ar{r}{\phi} &J\ar{d}{\lp_J}\\
K\times K \ar{r}[']{\phi \times K} &J \times K \ar{r}[']{J \times \phi} &J\times J,
\end{tikzcd}\]
which combined with the preceding computations induces a canonical isomorphism $\Psh^{\ms T_\phi} \al_J\cong\al_K \ms E^\phi $, as claimed.

Now we construct a quasi-inverse $\b_J$ to $\al_J$. Suppose we have a model $\bar M: \ms T_J^{op} \to \Set$. We want to find a $\mc C$-multiplicative $M: \ms T^{op} \to \ms E^J$ such that $\al_J M \cong \bar M$. For every $E: X \to I \times J$ in $\mc C_{/I\times J}$, we have a commutative diagram
\[\begin{tikzcd}
X\cart{dr}\ar{d} \ar{r}{E}&\ar{d}{I \times \lp_J} I \times J\ar{r}\opcart{dr} & I\times J\ar{d} \\
X\times J \ar{r}[']{E \times J} & I \times J \times J\ar{r}[']{\pi_{I \times J}} & I\times J
\end{tikzcd}\]
and a commutative diagram
\[\begin{tikzcd}
I\times J \cart{dr}\ar{r}\ar{d}[']{I \times \lp_J}& J\ar{d}{\lp_J} \\
I \times J \times J\ar{r}{\pi_J\times J}&J\times J
\end{tikzcd}\]

If $M$ is $\mc C$-multiplicative and $\al_J M = \bar M$ then we find that for $E \in \mc C_{/I\times J}$ and $T \in \ms T_I$ there is a natural isomorphism
\begin{align}\notag
\ms E^{I\times J}(E, M^I(T)) &\cong \ms E^{I\times J}(\ms C_{(\pi_I \cp E) \times J} \ms C^{E \times J} \ms C^{\pi_J \times J} \lp_J, M^I(T))\\
&\cong \ms E^{J \times J}(\lp_J, M^J(\ms T_{\pi_J}\ms T_{E} \ms T^{\pi_I \cp E} T))\notag\\
\label{mbyconstruction}&\cong \bar M(\ms T_{\pi_J} \ms T_E \ms T^{\pi_I \cp E}(T)).
\end{align}
Thus for a model $\bar M: \ms T_J^{op} \to \Set$ we define a $\mc C$-functor $M: \ms T^{op} \to \ms E^J$ with components
$M^I: \ms T_I^{op} \to \ms E^{I \times J} \cong \Sh^{\mc C_{/I\times J}}$ given by
\[M^I(T)(E) = \bar M(\ms T_{\pi_J\cp E} \ms T^{\pi_I \cp E}(T)).\]
To see that $M^I(T)$ is in fact a sheaf on $\mc C_{/I\times J}$, we recall that every $T$ in $\ms T_I$ is of the form $T = \tau^I F$ for some $F \in \mc C_{/I}$. But now consider the pasting of pullback squares
\[\begin{tikzcd}[column sep=large]
\cart{rd}.\ar{d}[']{\ms C^{\pi_I \cp E}F}\ar{r}{\ms C^{F \times J}  E}&.\cart{rd}\ar{d}{F \times J} \ar{r}& . \ar{d}{F}\\
.\ar{r}[']{E}&I\times J\ar{r}[']{\pi_I} & I.
\end{tikzcd}\]
We have a canonical isomorphism
\begin{align*}
\ms T_{\pi_J} \ms T_E \ms T^{\pi_I \cp E}(\tau^I F) & \cong \tau^J \ms C_{\pi_J} \ms C_E \ms C^{\pi_I \cp E}(F)\\
&\cong \tau^J\ms C_{\pi_J} \ms C_{F \times J}\ms C^{F \times J} E,
\end{align*}
and thus
\[M^I(\tau^I F)(E) \cong \bar M(\tau^J\ms C_{\pi_J} \ms C_{F \times J}\ms C^{F \times J} E).\]
Since $\bar M(\tau^J - )$ is a sheaf on $\mc C_{/J}$ by assumption, and all of the restrictions and left extensions of $\ms C$ are continuous morphisms of sites, it follows that the $\bar M$ we have defined is a sheaf on $\mc C_{I\times J}$.

Now we need to show that the $M^I$ we have defined are components of a $\mc C$-multiplicative $\mc C$-functor $M: \ms T^{op} \to \ms E^I$. Recall that for $\phi: I \to K$, the restriction and right extension along $\phi$ are defined for $P \in \ms E^{I \times J}$, $Q \in \ms E^{K \times J}$ by
\begin{align*}
(\ms E^{\phi \times J}P)(-) & = P(\ms C_{\phi\times J} -) \\
(\ms E_{\phi \times J*}Q)(-) & = Q(\ms C^{\phi\times J} -).
\end{align*}
Then we have
\begin{align*}
(\ms E^{\phi \times J} M^K(T))(E) & = \bar M(\ms T_{\pi_J \cp \phi \times J\cp E} \ms T^{\pi_K \cp \phi \times J \cp E}T)\\
&=\bar M(\ms T_{\pi_J \cp E} \ms T^{\phi \cp \pi_I \cp E}T),
\end{align*}
while
\begin{align*}
M^I(\ms T^{\phi} T) (E) & = \bar M(\ms T_{\pi_J \cp E} \ms T^{\pi_I \cp E} \ms T^\phi T)
\end{align*}
and thus the isomorphism $\ms T^{\phi \cp \pi_I \cp E} \cong \ms T^{\pi_I \cp E} \ms T^{\phi}$ induces an isomorphism 
\[M^I \ms T^\phi \cong \ms E^{\phi \times J} M^K.\]
Similarly, let $E \in \ms E^{K \times J}$ be given. The pullback diagram
\[\begin{tikzcd}
\ar{d}[']{(\phi \times J)^* E}.\ar{r}{E^*(\phi \times J)}&.\ar{d}{E}\\
\ar{d}[']{\pi_I}\ar{r}{\phi \times J}I \times J & K \times J\ar{d}{\pi_K}\\
\ar{r}{\phi}I & K
\end{tikzcd}\]
determines a Beck-Chevalley isomorphism 
\[\ms T^{\pi_K \cp E} \ms T^\phi \cong \ms T_{E^* (\phi\times J)} \ms T^{\pi_I \cp(\phi \times J)^* E},\]
and thus we have
\begin{align*}
M^K(\ms T_\phi T)(E) & = \bar M(\ms T_{\pi_J \cp E} \ms T^{\pi_K \cp E} \ms T_{\phi} T)\\
&\cong \bar M(\ms T_{\pi_J \cp E} \ms T_{E^*(\phi\times J)} \ms T^{\pi_I \cp (\phi \times J)^* E} T)\\
&\cong \bar M(\ms T_{\pi_J \cp (\phi \times J) \cp (\phi\times J)^* E} \ms T^{\pi_I \cp (\phi\times J)^* E}),\\
&\cong \bar M(\ms T_{\pi_J \cp (\phi \times J)^* E} \ms T^{\pi_I \cp (\phi \times J)^* E}T)\\
&= (\ms E_{\phi \times J*} M)^I(T).
\end{align*}
Thus $M$ is both $\mc C$-cartesian and $\mc C$-multiplicative. 

Let 
\[\b_J: \Mod_{\Linton}^{\ms T,J} \to \Mod_{\Lawvere}^{\ms T}(\ms E^J)\]
be the functor sending $\bar M$ to $M$. We claim that $\b_J$ is a quasi-inverse for $\al_J$. Indeed the isomorphism $\b_J \al_J \cong \id$ is the observation in~\eqref{mbyconstruction}, while the isomorphism $\al_J \b_J \cong \id$ is given by the formula
\begin{align*}
\al_J \b_J(\bar M) &\cong \bar M(\ms T_{\pi_J \cp \lp_J} \ms T^{\pi_J \cp \lp_J} T) \\
&\cong \bar M(T).
\end{align*}
Since each component of $\al_J$ is an equivalence, we conclude that $\al$ is an equivalence as well.
\qed\end{proof}

\section{Construction of $\mc E$-products}
\subsection{Extension of $\Mod^{\ms T}$ to $\mc E$}
If $\ms A$ is a fibered $\mc C$-category with $\mc C$-products, we can extend the $\mc C$-category $[\ms A,\ms E]_\tm$ of $\mc C$-multiplicative functors into a fibered $\mc E$-category corresponding to the pseudofunctor
\[[\ms A, \ms E]_\tm^P: =[\ms A, \ms E^P]_\tm\]
Since $\mc E$ is a topos, every map $\phi:P \to Q$ gives rise to a 
string of {\em fibered} adjunctions $\ms E_{\phi} \dashv \ms E^\phi \dashv \ms E_{\phi*}$. In particular, the functors $\ms E^\phi$ and $\ms E_{\phi*}$ preserve $\mc E$-products, so that the fibered adjunctions $\ms E^{\phi} \dashv \ms E_{\phi*}$ induce corresponding adjunctions $[\ms A, \ms E^\phi]_\tm\dashv [\ms A, \ms E_{\phi_*}]_{\tm}$ in $\CAT$. This implies that the $\mc E$-category $[\ms A, \ms E]_\tm$ is well-defined and inherits the property of having fibered $\mc E$-products from $\ms E$.

More generally, if $\ms B$ is any fibered $\mc C$-category, we have a natural equivalence 
\[[\ms B, [\ms A, \ms E]_\tm] \cong [\ms A,[\ms B,\ms E]]_\tm.\]
In particular, if $R$ is a covering sieve for $P \in \mc E$, we have
\begin{align*}
[R, [\ms A,\ms E]_\tm] & \cong [\ms A,[R, \ms E]]_\tm \\
&\cong [\ms A, [P,\ms E]]_\tm\\
&\cong [P,[\ms A,\ms E]_\tm],
\end{align*}
because $\ms E$ is a stack for the canonical topology of $\mc E$. Thus the fibered $\mc E$-category $[\ms A,\ms E]_\tm$ is also a stack over $\mc E$.

We can summarize the above discussion as follows:
\begin{theorem}
Let $\ms A$ be a category with $\mc C$-products. The $\mc C$-category $[\ms A, \ms E]_\tm$ is a stack over $\mc C$, and the extension to a stack over $\mc E$ has $\mc E$-products.
\end{theorem}
If $\tau: \ms C \to \ms T$ is a standard $\mc C$-ary theory, we can reframe this as a statement above $\Mod^{\ms T}$.
\begin{corollary}
The $\mc E$-category $y_* \Mod^{\ms T}$ is a stack over $\mc E$ and has $\mc E$-products.
\end{corollary}

\section{Colimit constructions}
\subsection{Modelification}
We use the set-theoretic notion of a Linton model to construct various colimits in categories of models. All of these constructions rely on the following fundamental fact. 

\begin{proposition}
Let $\mc C$ and $\mc T$ be small categories, and let $\tau: \mc C \to \mc T$ be a Linton theory. The inclusion
\[\Mod^{\mc T} \to \Psh^{\mc T}\]
admits a left adjoint.
\end{proposition}
\begin{proof}
There is a limit sketch on $\mc C$ whose models are sheaves. That is, there is a collection $\Phi$ of cocones in $\mc C$, such that $P \in \Psh^{\mc C}$ is a sheaf if and only if for each $\G \in \Phi$ the cone $P \G$ is a limit cone in $\Set$.

It follows immediately that $\Mod^{\mc T}$ is the category of models for a limit sketch on $\mc T$. In fact, a presheaf $M \in \Psh^{\mc T}$ is a model of $\mc T$ if, for every $\G \in \tau \Phi$, the cone $M \G$ is a limit cone.

The claim then follows from a well-known theorem about limit sketches.
\qed\end{proof}

\subsection{Adjoints to restriction functors}
Let $\mc T$ and $\mc T'$ be limit sketches (e.g. Linton theories). We say that $G: \mc T \to \mc T'$ is {\em sketchy} if the restriction functor $\Psh^G: \Psh^{\mc T'} \to \Psh^{\mc T}$ preserves models. 

 Thus if $G$ is sketchy we have a commutative diagram
\[\begin{tikzcd}
\Mod^{\mc T'}\ar{r}{\Mod^G}\ar[tail]{d} & \ar[tail]{d}\Mod^{\mc T}\\
\Psh^{\mc T'} \ar{r}{\Psh^G}& \Psh^{\mc T}
\end{tikzcd}\]
Since $\Psh^G$ always has a left adjoint $\Psh_{G}$ (left Kan extension), and the inclusion of models into presheaves is reflective, we may apply the adjoint lifting theorem to show that the restriction functor $\Mod^G$ admits a left adjoint $\Mod_{G}$. 

\subsection{The free model functor}
We will need the following lemma to construct the free model functor as a $\mc C$-fibered left adjoint to the forgetful $\mc C$-functor. 
\begin{proposition}
\label{leftadjoint}
Let $\ms S$ and $\ms T$ be standard $\mc C$-ary theories. If $H: \ms S \to \ms T$ is $\mc C$-additive, and each component $H^I$ is sketchy, then the restriction $\mc C$-functor 
\[\Mod^H: \Mod^{\ms T} \to\Mod^{\ms S}\]
has a $\mc C$-fibered left adjoint.
\end{proposition}
\begin{proof}
For each $I$ in $\mc C$, we have an adjunction
\[\Mod_{H_I} \dashv \Mod^{H_I}: \Mod^{\ms T_I} \rla \Mod^{\ms S_I}.\]
Since $H$ is opcartesian, we have for every $\phi: I \to J$ in $\mc C$, two canonical isomorphisms
\[\begin{tikzcd}[row sep=large, column sep=large]
\ms S_I
\ar{d}[']{\ms S_{\phi}}\ar{r}{H_I}& \ms T_I\ar{d}{\ms T_{\phi}}\\
\sar{ur}{\al_\phi}
\ms S_J\ar{r}[']{H_J}& \ms T_J
\end{tikzcd}\qq\AND\qq
\begin{tikzcd}[row sep=large, column sep=large]
\Mod^{\ms T_I}
\ar[leftarrow]{d}[']{\Mod^{\ms T_{\phi}}}\ar[leftarrow]{r}{\Mod^{H_I}}& \Mod^{\ms S_I}\ar[leftarrow]{d}{\Mod^{\ms S_{\phi}}}\sar{dl}{\Mod^{\al_\phi}}\\
\Mod^{\ms T_J}\ar[leftarrow]{r}[']{\Mod^{H_J}}&\Mod^{ \ms S_J},
\end{tikzcd}
\]
In order for $\Mod^H$ to have a fibered left adjoint, we must show that mate of $\Mod^{\al_\phi}$ under the adjunctions $\Mod_{H_I}\dashv \Mod^{H_I}$ and $\Mod_{H_J} \dashv \Mod^{H_J}$ is an isomorphism. 

In fact, the comate of $\al_\phi$ under the adjunctions $\ms S_{\phi} \dashv \ms S^\phi$ and $\ms T_{\phi} \dashv \ms T^\phi$ is a isomorphism, because $H$ is cartesian. It follows immediately that the comate of $\Mod^{\al_\phi}$ under the adjunctions $\Mod^{\ms S_{\phi}} \dashv \Mod^{\ms S^{\phi}}$ and $\Mod^{\ms T_{\phi}} \dashv \Mod^{\ms T^\phi}$ is an isomorphism. By a general property of mates, this implies that the mate of $\Mod^{\al_\phi}$ under $\Mod_{H_I} \dashv \Mod^{H_I}$ and $\Mod_{H_J}\dashv \Mod^{H_J}$ is an isomorphism as well.
\qed\end{proof}
\begin{corollary}
Let $\tau: \ms C \to \ms T$ be a standard $\mc C$-ary theory. The forgetful functor $U: \Mod^{\ms T} \to \ms E$ has a $\mc C$-fibered  left adjoint.
\end{corollary}

\subsection{Enlarging the sketch}
In order to show cocompleteness for the $\mc E$-category $y_*\Mod^{\ms T}$, we will use the comparison lemma~\cite[SGA 4.III.4]{ArtinGrothendieckVerdier1972} for categories of sheaves to avoid size issues due to the use of large sketches. We note that the ideas in this section are almost entirely due to Linton~\cite{Linton1969} and Giraud~\cite[II]{Giraud1971}.

\begin{lemma}\label{comparison}
Let $\tau: \mc C\to \mc T$ be a Linton theory. Suppose we have a diagram
\[\begin{tikzcd}
\car{dr}{\cong}\mc C\ar{r}{y_{\mc C}}\ar{d}{ \tau } & \mc E\ar{d}{\bar{ \tau }}\ar{dr}{F}&\ \\
\mc T\ar{r}{Z} & \bar{\mc T}\ar{r}{\io}&\Mod^{\mc T},
\end{tikzcd}\]
where $F$ is the left adjoint to the forgetful functor, the functor $\bar \tau$ is identity on objects, and the functors $Z$ and $\io$ are fully faithful. Then $\Mod^{Z}$ is an equivalence of categories.
\end{lemma}
\begin{proof}
Let $Z^* = \Mod^{Z}$. Then $Z^*$ has a {\em right} adjoint $R$, defined by
\[(RM)(\bar T) = \Mod^{\mc T}(\io \bar T, M),\]
which is essentially the right Kan extension along $Z$. For this to make sense, we must verify that $RM$ is a model whenever $M$ is. Indeed, we have
\[(RM)(\bar  \tau  X) \cong \Mod^{\mc T}(FX, M) \cong \mc E(X, UM),\]
which is a sheaf. 

The counit $Z^* R \to \id$ is an isomorphism, because $Z$ is full and faithful. To see that the unit is an isomorphism, we observe that
$\bar  \tau ^*: \Mod^{\bar {\mc T}} \to \Mod^{\mc E}$ is conservative because $\bar  \tau$ is identity on objects, and $y^*: \Mod^{\mc E} \to \Mod^{\mc C}$ is an equivalence by the comparison lemma for categories of sheaves. Thus the unit $\eta:\id \to RZ^*$ is an isomorphism if and only if this is true for the pasting diagram 
\[\begin{tikzcd}[column sep=large, row sep=large]
.\ar{dr}[name=R, pos=0.25]{R} & . \ar{r}{ \tau ^*}\car{dr}{\cong}&.&. \\
.\ar{u}{Z^*}\ar{r}[swap,name=id, pos=0.4]{\id} &.\ar{u}[']{Z^*}\ar{r}[']{\bar{ \tau }^*}&.\ar{u}[']{y^*}\arrow[Rightarrow,from=id, to=R, shorten <=5pt, shorten >= 10pt,"\eta"{xshift=-1ex}]
\end{tikzcd}\]
But the left-hand square in the diagram can be inverted by pasting with the counit of the adjunction. It follows that the indicated two-cell is invertible, and therefore the same is true of the counit $\e$, because $y^*\bar  \tau ^*$ is conservative.
\qed\end{proof}

Let $\tau: \ms C\to \ms T$ be a standard $\mc C$-ary theory, and let $F: \ms E \to \Mod^{\ms T}$ be the left adjoint to the forgetful functor. If we define $\bar {\ms C} = \ms E$, we may factor $F$ as 
\[\bar {\ms C} \rra{\bar\tau} \bar {\ms T} \rra{\bar\io} \Mod^{\ms T},\]
where $\bar{\tau}$ is identity on objects and $\bar\io$ is fully faithful.  An arrow of $\bar{\ms T}$ is cartesian or opcartesian iff this is true of the image under $\bar\io$. In particular, the $\mc C$-category $\bar{\ms T}$ is bifibered, and the functor $\bar \tau: \ms E \to \bar{\ms T}$ is bicartesian, because $F$ a fibered left adjoint and therefore bicartesian.

As $\Mod^{\ms T}$ is a stack over $\mc C$, the adjunction $F \dashv U$ in $\STACK^{\mc C}$ induces an adjunction $y_* F \dashv y_* U$ in $\STACK^{\mc E}$. We can interpret this as an adjunction 
\[\bar F \dashv \bar U: y_* \Mod^{\ms T} \rla \bbar{\ms C}.\]
Here $\bbar{\ms C} = \Arr(\mc E)$ is the fibered $\mc E$-category of arrows in $\mc E$, which is equivalent to $y_* \ms E$. By factoring $\bar F$ as 
\[\bbar{\ms C} \rra{ \bbar{\tau}} \bbar{\ms T}\rra{\bbar{\io}} y_* \Mod^{\ms T},\]
we obtain a Lawvere $\mc E$-theory. The topology on each site $\mc E^P$ is the canonical one, and since each functor $(y_*F)^P: (y_* \ms E)^P \to (y_* \Mod^{\ms T})^P$ is an ordinary left adjoint we find that all of the colimit cylinders in $\mc E^P$ are preserved by $\bbar{\tau}^P$. Although we have not shown that $y_* \Mod^{\ms T}$ has fibered sums, the functor $F$ does preserve the opcartesian arrows in $\bbar{\ms C}$, because it is a fibered left adjoint. It follows in particular that $\bbar{\ms T}$ is a bifibered $\mc E$-category.

\begin{proposition}
There is an equivalence of fibered $\mc E$-categories
\[\Mod^{\bbar{ \ms T}} \cong y_*\Mod^{\ms T}.\]
\end{proposition}
\begin{proof}
 In fact, we have an equivalence $\Mod^{\bbar{\ms T}} \cong [\bbar{\ms T}^{op}, y_*\ms E]_\times$, which shows that $\Mod^{\bbar{\ms T}}$ is a stack over $\mc E$. By the comparison lemma for stacks, it then suffices to show that $y^* \Mod^{\bbar{\ms T}} \cong \Mod^{\ms T}$ as fibered $\mc C$-categories. Now, one can check that
\[y^*\Mod^{\bbar{\ms T}}\cong \Mod^{y^*\bbar{\ms T}} \cong \Mod^{\bar{\ms T}},\]
and we have $Z: \ms T \to \bar{\ms T}$ fitting into a diagram of fibered $\mc C$-categories and functors 
\[\begin{tikzcd}
\car{dr}{\cong}\ms C\ar{r}{y_{\ms C}}\ar{d}{ \tau } & \ms E\ar{d}{\bar{ \tau }}\ar{rd}{F}&\ \\
\ms T\ar{r}{Z} & \bar{\ms T}\ar{r}{\bar \io}&\Mod^{\ms T}.
\end{tikzcd}\]
We claim that $\Mod^{Z}: \Mod^{\bar{\ms T}} \to \Mod^{\ms T}$ is an equivalence. Equivalently, all of the components $\Mod^{Z_I}: \Mod^{\bar{\ms T}_I}\to \Mod^{\ms T_I}$ are equivalences. But this follows from Lemma~\ref{comparison}, by applying the 2-functor $\ev_I: \FIB^{\mc C} \to \CAT$.
\qed\end{proof}

\begin{proposition}
The fibered $\mc E$-category $\Mod^{\bbar{\ms T}}$ is bifibered, and its fibers are locally presentable.
\end{proposition}
\begin{proof}
Let $\ms F$ be the opfibered $\mc E$-category defined by
\[\ms F_P = \mc C_{/P}.\]
Note that this is certainly not a fibered $\mc E$-category. We define an opfibered $\mc E$-category $\ms T'$ and opcartesian $\mc E$-functors $\mc \tau: \ms F \to \ms T'$ and $Z: \ms T' \to \bbar{\ms T}$ so as to obtain a diagram of opcartesian $\mc E$-functors
\[\begin{tikzcd}
\car{rd}{\cong}\ms F \ar{r}{y} \ar{d}{\tau}& y_* \ms E \ar{d}{\bbar \tau} \ar{rd}{F}&\ \\
\ms T' \ar{r}{Z} & \bbar{\ms T} \ar{r}{\io} & \Mod^{\ms T}
\end{tikzcd}\]
as above. By means of this diagram, we deduce that $\Mod^{Z}: \Mod^{\bbar{\ms T}} \to \Mod^{\ms T'}$ is an equivalence. But $\ms T'$ has small fibers. Thus the category of models $\Mod^{\ms T'_I}$ is locally presentable, and each restriction functor $\Mod^{\ms T_\phi'}$ has a left adjoint by Proposition~\ref{leftadjoint}.
\qed\end{proof}

\begin{corollary}
The $\mc E$-category $y_* \Mod^{\ms T}$ has $\mc E$-sums and $\mc E$-products, and its fibers are locally presentable.
\end{corollary}
\begin{proof}
Indeed, we have an equivalence $y_* \Mod^{\ms T} \cong \Mod^{\bbar {\ms T}}$. But $\Mod^{\bbar{\ms T}}$ has $\mc E$-products, since it is the category of models for a Lawvere $\mc E$-theory. It is also bifibered, as we have just shown, which means that it has $\mc E$-sums. Moreover, its fibers are locally presentable. All of these properties can be transported to $y_* \Mod^{\ms T}$.
\qed\end{proof}

\section{Monoidal structures}
In this section, we show that for a commutative theory, the category of models admits a well-behaved tensor product. The ideas in this section are essentially due to Day~\cite{Day1970,Day1974,Day1972} and Day and Street~\cite{DayStreet1995} (see also~\cite{BastianiEhresmann1972}).
\subsection{Change of variables}
An adjunction $\phi_!\dashv \phi^*: \mc J_1 \rla \mc J_2$ between two categories induces an adjunction $[\phi^*,\mc D]\dashv [\phi_!,\mc D]: [\mc J_1,\mc D] \rla [\mc J_2,\mc D]$ between functor categories, which we may express as a natural isomorphism
\[[\mc J_1,\mc D]([\phi^*,\mc D] F,G) \cong [\mc J_2,\mc D](F, [\phi_!,\mc D]G)\]
or more suggestively by the {\em change of variables formula}
\[\int_{J_1 \in \mc J_1} \mc D^{F\phi^* J_1}{}_{G_{J_1}} \cong \int_{J_2 \in \mc J_2} \mc D^{FJ_2}{}_{G\phi_! J_2},\]
which can also be seen as a special case of the mate correspondence. By separation of variables, we deduce a more general change of variables formula for ends, namely that for $H: \mc J_2^{op} \times \mc J_1 \to \mc D$ we have
\[\int_{J_1 \in \mc J_1} H^{\phi^* J_1}{}_{J_1} \cong \int_{J_2 \in \mc J_2} H^{J_2}{}_{\phi_! J_2}.\]
Similarly, for $H: \mc J_1^{op} \times J_2 \to \mc D$, we have an isomorphism of coends
\[\int^{J_2 \in \mc J_2} H^{\phi_! J_2}{}_{J_2} \cong \int^{J_1 \in \mc J_1} H^{J_1}{}_{\phi^* J_1}.\]
\subsection{Monoidal adjunctions}
In what follows, we will assume that all monoidal categories and functors are {\em symmetric} monoidal. 

We recall Day's {\em reflection theorem}~\cite{Day1972}. For a closed monoidal category $\mc V$ and a reflective subcategory $r\dashv \io:\mc W \rla \mc V$, the adjunction can be improved to a monoidal adjunction if and only if $\mc W$ is an {\em exponential ideal} in $\mc V$. Assuming that $\mc W$ is replete, this means  that $[V, W] \in \mc W$ whenever $W \in \mc W$ and $V \in \mc V$. Under these assumptions, the subcategory $\mc W$ is closed monoidal as well, and $\io$ is a closed functor.

\subsection{Day convolution}
We apply the reflection theorem with $\mc V = (\Psh^{\mc C}, \brot, y(e))$, where $(\mc C,\ot,e)$ is a small monoidal category and the {\em Day convolution product}~\cite{Day1970} of presheaves $P, Q \in \Psh^{\mc C}$ is the presheaf $P\brot Q$ defined by the coend
\[(P \brot Q)^C = \int^{C_1,C_2 \in \mc C} P^{C_1} \times Q^{C_2} \times \mc C^C{}_{C_1\ot C_2}.\]
The internal hom is the presheaf $[P,Q]$ defined by the end
\[[P,Q]^C =\int_{C_1 \in \mc C} [P^{\mc C_1},Q^{\mc C_1\ot \mc C_2}].\]
The Yoneda embedding $y: \mc C \to \Psh^{\mc C}$ is strong monoidal with respect to the given monoidal structure on $\mc C$ and the Day convolution on $\Psh^{\mc C}$. Note that when $\mc C$ is {\em cartesian} monoidal, the same holds for $\Psh^{\mc C}$, so this reduces to the fact that the Yoneda embedding preserves finite (indeed any) products.

\subsection{Monoidal sketches}
Let $(\mc S,\Phi)$ be a small sketch. We say that a monoidal structure $(\ot,e)$ on $\mc S$ is {\em compatible} if the tensor product functors $T\ot -: \mc S \to \mc S$ are all sketchy. In this case we say that $(\mc S, \Phi, \ot)$ is a {\em monoidal sketch}. 

The category $\Mod^{\mc S}$ is an exponential ideal in $\Psh^{\mc S}$.  Indeed, the category of models is closed under limits. For each $S_1$, the presheaf $S \mapsto N^{S_1 \ot S}$ is a model, because $S_1\ot-$ is sketchy and $N$ is a model. Now the Day internal hom $[M,N]$ is the limit, weighted by $M$, of the diagram $S_1 \mapsto N^{S_1 \ot -}$, which is valued in $\Mod^{\mc S}$. Thus $[M,N]$ is a model as well.

By Day's reflection theorem, the reflective inclusion $r\dashv i: \Mod^{\mc S} \into \Psh^{\mc S}$ enriches in an essentially unique way to a monoidal adjunction, and the monoidal category $\Mod^{\mc S}$ thus obtained is closed.

\subsection{Functoriality of Day convolution}
If $G: \mc S_1 \to \mc S_2$ is any oplax monoidal functor, then $\Psh^G: \Psh^{\mc S_2} \to \Psh^{\mc S_1}$ is lax monoidal with respect to the Day convolution structure. The adjunctions $r_i \dashv \io_i: \Mod^{\mc S_i} \dashv \Psh^{\mc S_i}$ both live in the 2-category $\MonCAT_{lax}$. Thus, if $G$ is also sketchy, we obtain a lax monoidal structure on $\Mod^{G} = r_1 \Psh^G \io_2$ as well. Similarly, a monoidal transformation from $G$ to $G'$ induces a corresponding monoidal transformation from $\Mod^{G'}$ to $\Mod^{G}$. If we define $\MonSketch$ to be the 2-category of small monoidal sketches, sketchy oplax monoidal functors, and monoidal transformations, then $\Mod$ can be defined as a 2-functor
\[\Mod: \MonSketch^{coop}\to \MonCAT_{lax},\]
which takes values in the 2-category monoidal categories, lax monoidal functors, and monoidal transformations. Note that even if $G$ is a {\em strong} monoidal functor,  the functor $\Psh^G$ described above is usually only lax.

\subsection{Monoidal $\mc C$-categories}
The category $\FIB^{\mc C}$ has finite products, and we may define a monoidal (fibered) $\mc C$-category to be a pseuomonoid $\ms F$ in the monoidal 2-category $(\FIB^{\mc C},\times, 1)$. This determines 2-categories $\MonFIB_{lax/oplax/strong}^{\mc C}$ of monoidal $\mc C$-categories, lax/oplax/strong monoidal cartesian functors, and monoidal cartesian transformations.

\subsection{The external product}
Shulman~\cite{Shulman2008} has shown that the definition of monoidal $\mc C$-category we have given is equivalent to a somewhat different one. He defines a {\em monoidal fibration} to be a fibred $\mc C$-category  $p: \int \ms F \to \mc C$, such that $p$ is strict monoidal and  each functor $T \ot - $ preserves cartesian arrows. The model case is when $\ms F = \ms C$, so that $\int \ms C$ is the $\mc C$ category of arrows $\cod: \Arr(\mc C) \to \mc C$. For arrows $f_i:X_i \to I_i$, we have an {\em external product} $f_1 \times f_2: X_1 \times X_2 \to I_1 \times I_2$, and the projection satisfies $\cod(f_1 \times f_2) = \cod f_1\times \cod f_2$.

Given an external product $\boxtimes: \int \ms F \times \int \ms F \to \int \ms F$ making $\ms F$ into a monoidal fibration, the {\em internal product} $\ot_I$ on the fiber $\ms F^I$ is obtained by taking
\[A \ot_I B = \ms F^{\lp_I} (A \boxtimes B),\]
where $\lp: I \to I \times I$ is the diagonal. If $\mc C$ is cartesian, then this sets up an equivalence between monoidal fibrations and monoidal $\mc C$-categories. A monoidal functor $F: \ms F \to \ms G$ is then an ordinary monoidal functor, such that the identity $p_{\ms G} F = p_{\ms F}$ is a monoidal transformation.

\subsection{Distributive $\mc C$-categories}
If $\mc V$ is a closed monoidal category, then the functors $V \ot - : \mc V \to \mc V$ are left adjoints, and must preserve colimits. Suppose $\mc C$ is a monoidal category which is not necessarily closed, and $y: \mc C \to \mc E$ is an embedding which preserves colimits in some set $\Phi$. We want to extend the monoidal structure on $\mc C$ to a closed monoidal structure on $\mc E$. It is clear that this will not be possible unless the functors $C \ot - :\mc C \to \mc C$ preserve colimits in $\Phi$ as well. This condition is precisely the compatibility condition on a monoidal sketch. For a fibered $\mc C$-category, we also have colimits which are indexed by arrows in $\mc C$, and for an bifibered sketch we need a compatibility condition with respect to these colimits as well.

Let $\ms T$ be a monoidal $\mc C$-category which is bifibered, and let $\ms T_\phi$ be a left extension functor for $\phi: I \to J$. Given $A, B \in \ms T^I$, the tensor product of the opcartesian arrows $A \to \ms T_\phi A$ and $B \to \ms T_\phi B$ over $\phi$ is an arrow $A \ot_I B \to \ms T_\phi A \ot_J \ms T_\phi B$ over $\phi$. By the universal property of left extensions, we may define a canonical map
\[\ms T_\phi(A \ot_I B) \to \ms T_\phi A \ot_J \ms T_\phi B\]
and similarly a canonical map $\ms T_\phi e_I \to e_J$. These give every left extension a canonical oplax monoidal structure. Given a restriction $\ms T^\phi$, which has a canonical strong monoidal structure, we can define a canonical comparison map
\[\ms T_\phi(\ms T^\phi T' \ot_I T) \to \ms T_\phi \ms T^\phi T' \ot_J \ms T_\phi T \to T' \ot_J \ms T^\phi T,\]
and we say that the projection formula holds if this comparison map is an isomorphism.
\begin{definition}
Let $\ms T$ be a monoidal bifibered $\mc C$-category. We say that $\ms T$ is {\em distributive} if the projection formula
\[ \ms T_{\phi}(\ms T^\phi T' \ot_I T)\cong T' \ot_J \ms T_{\phi} T \]
holds for all $T' \in \ms T^J$, $T \in \ms T^I$ and $\phi: I \to J$.
\end{definition}

\subsection{Structured sketches}
In this section we fix some terminology regarding sketches with extra structure. A $\mc C$-(op/bi)fibered sketch is $\mc C$-(op/bi)fibered category whose fibers are sketches, such that the restrictions (left extensions/both) are all sketchy. A monoidal $\mc C$-sketch is a fibered $\mc C$-sketch whose fibers are monoidal sketches. A distributive $\mc C$-sketch is a monoidal bifibered $\mc C$-sketch which is distributive.

\subsection{Tensor product of models}
Suppose $\ms T$ is a distributive $\mc C$-sketch. For each fiber we have a monoidal structure $(\ot_I, e_I)$ given by the Day convolution. The functors $(\ot_I,e_I)$ are components of cartesian functors $\ot: \Mod^{\ms T} \times \Mod^{\ms T} \to \Mod^{\ms T}$ and $e: 1 \to \Mod^{\ms T}$.

To see this, suppose we are given $\phi: I \to J$ and left extension and restriction functors $\ms T_\phi$ and $\ms T^\phi$. In the adjunction $\ms T_\phi \dashv \ms T^\phi$, the unit and counit are monoidal transformations, where we consider $\ms T_\phi$ and $\ms T^\phi$ with their canonical oplax monoidal structures. It follows that in the induced adjunction $\Mod^{\ms T_\phi} \dashv \Mod^{\ms T^\phi}$, the functors are lax monoidal and the unit and counit are monoidal transformations. By Kelly's doctrinal adjunction theorem~\cite{Kelly1974}, we can conclude that the left adjoint $\Mod^{\ms T_\phi}$ is {\em strong} monoidal. Thus we have provided a canonical strong monoidal structure for every restriction. Moreover, the pseudofunctoriality constraints of $\Mod^{\ms T}$ are induced by the pseudofunctoriality constraints of an opcleavage for $\ms T$ and thus inherit the monoidal coherence.

In fact, each restriction $\phi^* = \Mod^{\ms T,\phi}$ is actually a {\em closed} functor. For $\phi: I \to J$ and $M,N \in \Mod^{\ms T, I}$, we have 
\begin{align*}
(\phi^*[M,N]_I)^T & = [M,N]_I^{\phi _!T} \\
&= \int_{T' \in \ms T^I} [M^{T'},N^{T'\ot_J \phi_! T}] \\
&\cong \int_{T' \in \ms T^I} [M^{T'},N^{\phi_!(\phi^*T'\ot T)}] 
\end{align*}
by the projection formula. But then by change of variables this is canonically isomorphic to
\[\int_{T' \in \ms T^J} [M^{\phi_! T'}, N^{\phi_!(T' \ot_I T)}]=[\phi^* M, \phi^* N]_J^T.\]

\subsection{Commutative theories}
Recall that if $\mc C$ has finite limits, then  $\ms C=\Arr(\mc C)$ is a cartesian distributive $\mc C$-category. We say that a (standard $\mc C$-ary) theory $\tau: \ms C \to \ms T$ is {\em commutative} if $\tau$ enriches to a strict monoidal $\mc C$-functor, where the monoidal structure on $\ms C$ is cartesian. The monoidal $\mc C$-category $\ms T$ inherits distributivity  from $\mc C$.
\begin{lemma}
\label{tenissketchy}
Let $\tau: \ms C \to \ms T$ be a commutative theory. Then $\ms T$ is a distributive $\mc C$-category. For $f: I \to J$ in $\ms C^J$ and $T \in \ms T^J$, we have a natural isomorphism
\[\tau^J f \ot_J T \cong \ms T_{f} \ms T^f T.\]
In particular, each functor $T \ot_J -: \ms T^J \to \ms T^{J}$ is sketchy, and $\ms T$ is a distributive $\mc C$-sketch.
\end{lemma}
\begin{proof}
We need to check that the projection formula holds. Each component of $\tau$ is identity on objects, and we need to show that for each $\phi: I \to J$ in $\mc C$, $E' \in \mc C_{/J}$ and $E \in \mc C_{/I}$, the canonical map
\[\ms T_\phi(\ms T^\phi \tau^J E' \ot_I \tau^I E) \to \tau^J E' \ot_J \ms T_\phi \tau^I E\]
is an isomorphism. But since $\tau$ is strict monoidal and bicartesian, this essentially reduces to the projection formula for $\ms C$, which is standard.

A consequence of the projection formula is that
\begin{align*}
\tau^Jf \ot_J T & \cong \tau^J(\ms T_{f} \id_I) \ot_J T\\
&\cong \ms T_{f}(\tau^I \id_I) \ot_J T\\
&\cong \ms T_{f}(\tau^I \id_I \ot_I \ms T^f T)\\
&\cong \ms T_{f} \ms T^f T,
\end{align*}
again since $\tau$ is bicartesian and strict monoidal.

In particular, the functors $T \ot_J -$ are sketchy because $\ms T_f$ and $\ms T^f$ are sketchy, and every $T$ is of the form $\tau^J f$ for some $f$.
\qed\end{proof}
From this lemma we see that the monoidal structure of $\ms T$, if it exists, is determined by $\tau$. Thus the commutativity assumption essentially means that operations commute.

\subsection{Free models}
Let $\tau: \ms C \to \ms T$ be a commutative Lawvere theory. The free model functor $\Mod_\tau$ is obtained by the composition
\[\Sh^{\ms C_I} \rra{\io} \Psh^{\ms C_I} \rra{\Psh_{\tau_I}} \Psh^{\ms T_I} \rra{r} \Mod^{\ms T_I}.\]
The Day convolution on $\Psh^{\ms C_I}$ is easily seen to be the cartesian product. Since the cartesian product is preserved by sheafification, the resulting monoidal structure on $\Sh^{\ms C_I}$ is {\em also} the cartesian product. Since the inclusion of sheaves into presheaves is a right adjoint, it preserves the cartesian product, and thus the first functor in this sequence is strong monoidal.

Since $\tau_I$ is strong monoidal, the left Kan extension $\Psh_{\tau_I}$ is strong monoidal with respect to the Day convolution~\cite{DayStreet1995}. Finally, the reflection $r$ sending presheaves into models is strong monoidal, because that is how we defined the monoidal structure on models. Thus the free model functor $\Mod_{\tau_I}$ is strong monoidal.

The reader can verify that each functor in this sequence is the component of a strong monoidal $\mc C$-functor.

\subsection{Extension to $\mc E$}
Let $\tau: \ms C \to \ms T$ be a commutative Lawvere theory. We have shown that $\ms T$ is a sketchy distributive fibration, and therefore gives rise to a monoidal $\mc C$-cosmos $\Mod^{\ms T}$. We know that $y_*\Mod^{\ms T}$ is the category of models for a large theory $\bbar{\tau}: \ms E \to \bbar{\ms T}$, and $\bbar \tau$ can be made strict monoidal using the fact that the free model functor is strong monoidal. Formally, it follows that $\Mod^{\bbar{\ms T}}$ admits an essentially unique closed monoidal structure extending that of $\Mod^{\ms T}$.

To avoid size issues and excessive abstraction, we can simply describe this structure explicitly. The objects of $y_* \Mod^{\ms T,P}$ are cartesian functors (pseudonatural transformations) from $P$ to $\Mod^{\ms T}$. Given $M,N \in y_* \Mod^{\ms T,P}$, the tensor product is defined pointwise for $f \in P^I$ by
\[(M \ot_P N)(f) = M(f) \ot_I N(f),\]
and the internal hom by
\[[M,N]_P(f) = [M(f),N(f)]_I.\]
These expressions are pseudonatural in $f$ because $M$ and $N$ are pseudonatural in $f$ and the restrictions in $\Mod^{\ms T}$ are strong monoidal and closed. The unit and counit for the adjunction $M\ot_P - \dashv [M,-]$ have as components the unit 
\[\ins_{M(f)}: N(f) \to [M(f),M(f)\ot_I N(f)]_I \]
and counit
\[\ev_{M(f)}: M(f) \ot_I [M(f),N(f)]_I \to N(f)\]
of the adjunction $M(f) \ot_I-\dashv [M(f),-]_I$. 

We recall the notion of an $\mc E$-cosmos.
\begin{definition}[\cite{Shulman2013}]
An $\mc E$-cosmos is a monoidal $\mc E$-category $\ms V$ such that 
\begin{enumerate}
\item The $\mc E$-category $\ms V$ has $\mc E$-sums and $\mc E$-products.
\item The fibers of $\ms V$ have small limits and colimits.
\item The fibers of $\ms V$ are closed.
\item The restriction functors closed.
\end{enumerate}
\end{definition}
\begin{theorem}
Let $\tau: \ms C \to \ms T$ be a commutative Lawvere theory. Then the extension to $\mc E$ of the category of models is an $\mc E$-cosmos.
\end{theorem}
\begin{proof}
The completeness and cocompleteness properties are true for any Lawvere theory. We have defined a closed monoidal structure on the fibers of $y_* \Mod^{\ms T}$, and it is straightforward to check that this gives $\ms V$ the structure of a monoidal $\mc E$-cosmos.
\qed\end{proof}

\section{Examples}
\subsection{Finitary algebraic theories}
Let $\mc C = \aln$ be a skeletal category of finite sets and functions. We can define a topology on $\mc C$ whose cylinders correspond to all finite sums $m = m_1 +\dotsb + m_k$. We then have $\mc E = \Set$. Every ordinary small category $\mc T$ corresponds to a small fibered $\mc C$-category $\ms T$, where $\ms T^n = [n,\mc T]$. A cartesian functor $\tau: \ms C \to \ms E$ is bicartesian precisely when $\tau^1: \mc C \to \mc T$ preserves finite sums. Thus we recover the usual notion of a Lawvere theory. 

If $\mc D$ is an ordinary category with finite products, then a multiplicative functor $M: \ms T^{op} \to \ms D$ is the same as an functor $M: \mc T^{op} \to \mc D$ preserving finite products in the ordinary sense. In particular, if $I$ is any set, then the category $\Mod^{\ms T, I}$ is equivalent to the category of models of the ordinary Lawvere theory $\mc T$ in $\Set^I$.

\subsection{Spaces with Lebesgue integration}
Our original motivation was to define a category of linear spaces which have a well-behaved notion of Lebesgue integral. One would imagine that if $V$ is a reasonable linear space and $c: \R \to V$ is a bounded measurable function, then for any finite measure $\mu$, there is a well-defined integral
\[\int c(x)\,d\mu(x) \in V.\]
Unless  $V$ is a separable Banach space, there does not seem to be a reasonable abstract condition on $c$ which ensures that the integral exists. For example, it is not enough, in general, for $c$ to be weakly measurable.

Intuitively, the space of measures on a Borel space $X$ is the free Banach space generated by $X$. However, there are too many linear maps in the category of Banach spaces for this to be true. For example, let $M(\R)$ be the space of Radon measures on $\R$. This is a Banach space with respect to the total variation norm. Let
\[\d: \R \to M(\R)\]
be the Dirac function, sending $x \in \R$ to the Dirac mass at $x$. We will see that $\d$ is {\em not weakly measurable} with respect to the total variation norm. This is connected to the failure of $\sli(\R,\R)$ to be the dual of $M(\R)$. In symbols, there is a mismatch
\[\Ban(M(\R),\R) \neq \sli(\R,\R)\]
between the space of bounded linear functionals on $M(\R)$ and the space of bounded linear maps $\R \to \R$. This can be contrasted with the coincidence
\[\Ban(M_c(\R),\R) = \ell^\infty(\R,\R),\]
where $M_c(\R)$ is the space of {\em countably-supported} finite measures on $\R$ and $\ell^\infty(\R,\R)$ is the space of not-necessarily-measurable bounded maps from $\R$ to $\R$. These non-measurable functions are a source for some of the bad linear functionals on $M(\R)$.

This situation may be rectified by considering instead the locally convex space $M_w(\R)$ which coincides as a vector space with $M(\R)$ and has the topological dual
\[\Lcvx(M_w(\R),\R) = \sli(\R,\R).\]
However, it is readily apparent that the bounded sets in $M_w(\R)$ are the same as the bounded sets in $M(\R)$, so we are now in the pathological situation where the bounded functionals and the continuous functionals do not coincide.

Our idea is to avoid all of these topological and measure-theoretic complications by taking the collection of bounded measurable curves $c: \R \to V$ as part of the intrinsic data of a linear space $V$. More generally, we will have a related notion of a {\em nonlinear} space equipped with curves which we imagine to be bounded and measurable. A bounded measurable map between spaces $X$ and $Y$ will then be a map $f: X \to Y$ which preserves boundedness and measurability of curves, but not necessarily any other structure. Our categories $\Lin$ and $\Space$ of linear and nonlinear spaces are designed with the express purpose of having (by definition) an isomorphism
\[\Lin(M(\R),V) \cong \Space(\R,V)\]
for any linear space $V$. This allows us to conceptualize the space of Radon measures as the ``free linear space'' generated by the points of $\R$. Moreover, we will be able to replace $\R$ by any nonlinear space $X$ and obtain a free linear space $X$ such that
\[\Lin(M(X),V) \cong \Space(X,V).\]
We will not use any topological or analytic miracles besides for the existence and good properties of Lebesgue measure on the real line. That is, we will take the theory of ``Lebesgue integration'' as a wholesale replacement for the theory of vector spaces. 

\subsection{Lextensive categories}
Let $\mc C$ be any category with finite limits and $\k$-ary disjoint sums which are stable under pullback. Then there is an $\k$-extensive topology $\Phi_{\k}$ on $\mc C$ generated by the $\k$-ary sums $\sum_{i \in \ld} I_i$, with $I_i \in \mc C$ and $\ld < \k$. If $\tau: \ms C \to \ms T$ is a Lawvere theory, then a model of $\ms T^I$ is precisely a functor $M: (\ms T^I)^{op} \to \Set$ preserving $\k$-ary products. 

In our main application, we take $\mc C$ to be the category of standard Borel spaces (see~\cite{Kuratowski1948,Mackey1957}) For example, let $\Borel$ be the category of standard Borel spaces, whose objects are
\[\Borel = \{0,1,2,\dotsc,\N, \R\},\]
each considered as a measurable space with the $\si$-algebra of Borel sets. The maps are the Borel-measurable functions. This category has countable limits and countable sums which are disjoint and stable under pullback, so it is a countable lextensive category.
 
If we consider $\Borel$ to be a {\em finitary} extensive site, then sheaves on $\Borel$ include interesting spaces like $\R_b$, where
\[\R_b^X = \sli(X, \R)\]
is the set of bounded measurable functions from $X$ to $\R$. It is clear that the presheaf $\R_b$ preserves finite products but not countable products.

\subsection{Integral kernels}
Let $Y$ be a standard Borel space, and let $M(Y)$ be the space of signed Radon measures on $Y$. Recall that a function $k: X \to M(Y)$ is a {\em measurable kernel} if, for each Borel set $E$, the map $x \mapsto k(x,E)$ is Borel. By standard results on measurable kernels (see, e.g.~\cite[Chapter 1, Lemma 1.7]{Dynkin1961}), the composition of kernels is as well-behaved as one might expect. We can thus define a fibered category $\FreeLin$ of free linear spaces, whose objects are measurable bundles of free linear spaces generated by the standard Borel spaces.

We first define the fibers of $\FreeLin$. If $f_i: X_i \to I$ are Borel spaces over $I$, then a bundle map 
\[k: X_1 \to MX_2\]
is a measurable kernel $k: X_1 \to MX_2$, such that $\sup_{x_1} \n{ k(x_1)} < \infty$ and
\[ k(x_1,x_2) = 1_{X_1\times_I X_2}(x_1,x_2)  k(x_1,x_2).\]
That is, the kernel $ k$ is supported on the Borel set 
\[X_1 \times_I X_2 = \{(x_1,x_2) \in X_1 \times X_2: f_1(x_1) = f_2(x_2)\}.\]
We can compose kernels $k: X_1 \to MX_2$ and $k': X_2 \to MX_3$ by the rule
\[ (k'\cp k)(x_1,x_3) = \int_{X_2}  k(x_1, x_2)\, k'(x_2,x_3)\,dx_2,\]
and the identity for this composition is the Dirac kernel $\d: X \to MX$ sending $x$ to $\d_x$.

It is easy to see that if $k'$ and $k$ are bundle maps, then the composition $ k'\cp k$ is a bundle map as well. We define $\FreeLin_I$ to be the category whose objects are Borel spaces over $I$ and whose arrows are bundle maps.

More generally, if $f_i: X_i \to I_i$ are bundles of Borel spaces, and $\phi: I_1 \to I_2$ is any Borel map, then say that $k: X_1 \to MX_2$ is a bundle map over $\phi$ if it is a bundle map with respect to the maps to $I_2$. We will notate this situation by
\[\begin{tikzcd}
X_1 \ar{r}{k}\armd & MX_2\armd\\
I_1\ar{r}{\phi}&I_2.
\end{tikzcd}
\]
Thus we obtain a split opfibration over $\Borel$. For each $\phi: I \to J$ and $f: X \to I$ we have a canonical opcartesian arrow $\d: X \to MX$  over $\phi$.

\subsection{Pushforward of measures}
If $h: X \to Y$ is a Borel map, then family of Dirac masses
\[M(h)(x,y) = \d(h(x),y)\]
is a measurable kernel $M(h): X \to M(Y)$. Thus $M$ defines a functor
\[M: \Borel \to \FreeLin_1,\]
Suppose now that 
\[\begin{tikzcd}
X \ar{r}{h}\ar{d}{f} & Y\ar{d}{g}\\
I\ar{r}{\phi} & J
\end{tikzcd}\]
is a commutative diagram of Borel maps. One can check that $M(h): X \to M(Y)$ is a bundle map over $\phi$. Thus $M$ extends to a $\Borel$-functor $\tau: \Arr(\Borel) \to \FreeLin$. The functor $\tau$ preserves opcartesian arrows simply because the functor $M$ preserves identities.

We claim that $\tau$ preserves cartesian arrows as well. Suppose given a cartesian diagram 
\[\begin{tikzcd}
X\times_I J \ar{d}{p_J}\ar{r}{p_X} & X\ar{d}{g}\\
J\ar{r}{\phi} & I
\end{tikzcd}\]
of Borel spaces, and let 
\[\begin{tikzcd}
Z \ar{r}{k}\armd&\armd M X\\
J \ar{r}{\phi}& I
\end{tikzcd}\]
be any bundle map over $\phi$ for some Borel map $f: Z \to J$. We need to check that $k$ factors uniquely through the bundle map 
\[\begin{tikzcd}
X \times_I J \ar{r}{M p_X}\armd&M X\armd\\
J  \ar{r}{\phi}&I.
\end{tikzcd}\]
Suppose that $\ti k: Z \to M(X\times_I J)$ is a bundle map. Then, for each $z \in Z$, the measure $\ti k(z)$ is supported on the set $\{(x,j) \in X\times_I J: g(z) = j\}$, and must be of the form

\[\ti k(z,x) = a(z,x) \ot \d(g(z), j),\]
where $a(z,x)$ is a kernel determined by 
\[a(z,x) = \int_J \ti k(z,x,j) \,dj = Mp_X \cp \ti k(z,x).\]
This shows that $\ti k$ is uniquely determined by the factorization $k = Mp_X \cp \ti k$. Now, given $k: Z \to MX$, define the kernel $\ti k: Z \to M(X \times_I J)$ by
\[\ti k(z,x, j) = k(z,x) \ot_I \d(f(z), j),\]
where  for measures $\mu \in M(X)$ and $\nu \in M(J)$, the notation
\[d\mu(x)\ot_I d\nu(y) = \eval{d(\mu\times \nu)(x,y)}_{X \times_I J}\]
indicates the restriction to $X\times_I J$ of the product measure. Then we can compute
\begin{align*}
(Mp_X \cp \ti k)(z,x) & = \int_{X\times_I J} k(z,x') \ot\d(f(z),j')\,\d(x',x)\,dx'\,dj'\\
&=\int_{X \times J} 1_{\phi(j') = g(x')} 1_{j' = f(z)} k(z,x')\ot \d(f(z),j') \d(x',x)\,dx'\,dj'\\
&=\int_{X\times J} k(z,x') \ot \d(f(z),j') \d(x',x)\,dx'\,dj'
\end{align*}
because $k$ is a bundle map and satisfies the identity
\[1_{\phi(j') = g(x')} 1_{j' = f(z)} k(z,x') = 1_{\phi(f(z))=g(x')} k(z,x') = k(z,x').\]
But the last integral is just $k(z,x)$, so $k$ factors as $k = M p_X \cp \ti k$, as desired.

We have shown that $\tau$ preserves precartesian arrows. Because $\tau$ is identity on objects, it follows that there are enough precartesian arrows that $\FreeLin$ is a fibered $\Borel$-category. Moreover, the functor $\tau$ is bicartesian.

\subsection{Tensor product of measures}
We have shown that the identity-on-objects functor $\tau: \Arr(\Borel) \to \FreeLin$ is bicartesian. For each $g: Y \to I$, one checks that the functor sending a space $f: X \to I$ in $\Borel_{/I}$ to the collection of bundle maps $k: X \to MY$ is a sheaf (for the finitary extensive topology) and thus $\tau$ is a Lawvere theory. 

The theory $\tau$ is moreover {\em commutative} in the sense that the underlying functor $\tau: \Arr(\Borel) \to \int \FreeLin$ is strict monoidal. That is, for spaces $f_i: X_i \to I_i$ and $g_i: Y_i \to J_i$ and bundle maps $k_i:X_i \to MY_i$ over maps $\phi_i: I_i \to J_i$, we can extend (any given) cartesian tensor product on the arrow category $\Arr(\Borel)$ to a tensor product
\[\begin{tikzcd}[column sep=large]
X_1 \times X_2 \armd\ar{r}{k_1 \ot k_2} & M(Y_1\times Y_2)\armd\\
I_1 \times I_2 \ar{r}{\phi_1 \times \phi_2}&J_1\times J_2
\end{tikzcd}\]
on the total category of $\FreeLin$. The monoidal structure is defined in terms of product measures. For each $(x_1,x_2) \in X_1\times X_2$, we let
\[k_1\ot k_2(x_1,x_2,y_1,y_2) = k_1(x_1,y_1) \ot k_2(x_2,y_2),\]
which defines a measure on $Y_1 \times Y_2$ in the usual way. The kernel thus defined is a bundle map over $\phi_1\times \phi_2$. If $k_i(x_i,y_i) = \d(h_i(x_i), y_i)$ for Borel maps $h_i: X_i \to Y_i$ over the $\phi_i$, then 
\begin{align*}
k_1\ot k_2(x_1,x_2,y_1,y_2) & = \d(h_1(x_1),y_1) \ot \d(h_2(x_2),y_2)\\
&=\d((h_1\times h_2)(x_1,x_2), (y_1,y_2)).
\end{align*}
This shows that the functor $\tau$ is strict monoidal, so the corresponding Lawvere theory is commutative by our definition.

\subsection{Concrete spaces}
A sheaf $F$ over the extensive site $\aleph_0$ is  determined up to isomorphism by its underlying set, because we have  isomorphisms $F(n)\cong F(1)^n$, natural in $n$, and the naturality condition then determines $F(f)$ for any $f$ in $\aln(n,m)$. 

For a sheaf over $\Borel$, on the other hand, we have a canonical map
\[F(X) \to F(1)^{\a{X}},\]
where $\a{X}$ is the set of points $p: 1 \to X$. However, this map will almost never be an isomorphism. On the other hand, it may happen that it is at least injective for each $X$, and in this case we can view the sheaf $F$ as a set $F(1)$ together with a choice, for each $\a{X}$, of a set of admissible curves $c: \a{X} \to F(1)$. A map of spaces $f: F \to G$ is then a map $f_1: F(1) \to G(1)$ between the underlying sets which that sends admissible curves in $F$ to admissible curves in $G$. We will call such a space a concrete space.

Similarly, given a Lawvere theory $\tau: \Arr(\Borel) \to \FreeLin$, we will say that a model $M$ of $\tau$ is concrete when the underlying sheaf $M\tau$ is concrete.

\section{Some linear spaces}
In what follows we define
\[\Space = \Sh(\Borel)\]
to be the topos of sheaves on the finitary extensive Borel site, and
\[\Lin = \Mod(\FreeLin)\]
to be the category of models for the theory of linear spaces defined above.

\subsection{Separable Banach spaces}
In a separable Banach space $V$, there is a well-behaved notion of measurability. If $X$ is a standard Borel space, a function $f: X \to V$ is said to be measurable precisely when $l\cp f: X \to \R$ is measurable for every bounded linear functional $l: V\to\R$. If in addition $f$ is bounded in norm, the Bochner integral
\[\int f(x) \,d\mu(x)\]
exists for any finite measure $\mu$ and is determined by the requirement that 
\[l \sps{\int f(x) \,d\mu(x)}= \int l(f(x))\,d\mu(x)\]
for every $l$ in the Banach space dual $V'$ of $V$. Let $\ti V(X)$ denote the set of measurable bounded functions from $X$ to $V$. If $k:Y \to MX$ is a kernel, then we define an action $\ti V(k): \ti V(X) \to \Set(Y,V)$ by the Bochner integral
\[\ti V(k)[f] (y) = \int f(x)\,k(y,x) \,dx.\]
The function $\ti V(k)[f]$ is a measurable function of $Y$ because for $l \in V'$ we have
\[l(\ti V(k)[f](y)) = \int l(f(x)) \,k(y,x)\,dx,\]
and this is a composition of measurable kernels because $x \mapsto l(f(x))$ is measurable. The function $\ti V(k)f$ is also bounded by the triangle inequality for the Bochner integral. Thus we have a map 
\[\ti V(k): \ti V(X) \to\ti V(Y),\]
and it is easy to check that the assignment $X\mapsto \ti V(X)$ defines a product-preserving functor from the category $\FreeLin_1^{op}$ to $\Set$. 

Now suppose that $T: V \to W$ is an bounded linear map between two separable Banach spaces. Then composition with $T$ clearly preserves the measurable curves, and we define
\[\ti T: \ti V \to \ti W\]
to be composition with $T$. This $\ti T$ is a natural transformation, because Bochner integration commutes with bounded linear maps. Thus we obtain a functor 
\[B: \SepBan \to \Lin\]
on the category $\SepBan$ of separable Banach spaces, sending each space  $V$ to the linear space $\ti V$ defined above.

\begin{proposition}
The functor
\[\ti{(-)}: \SepBan \to \Lin\]
sending a separable Banach space $V$ to the linear space
\[\ti V(X) = \{c: X \to V\mid \text{$c$ is measurable and } \sup_x \n{c(x)} <\infty\}\]
is full and faithful.
\end{proposition}
\begin{proof}
Faithfulness is obvious, so it remains to show that the functor is full. Suppose that $\ti T: \ti V \to \ti W$ is a linear space map. Since $\ti V$ and $\ti W$ are concrete, the map $\ti T$ is determined by its action on the underlying sets. Since $\ti T$ commutes with integration, it preserves finite linear combinations, which means it is linear. Moreover, it is bounded, because the image of any bounded sequence of points must be a bounded sequence of points.
\qed\end{proof}

\subsection{Spaces of measures}
The Yoneda embedding gives some trivial examples of linear spaces, namely the spaces $MY$ of Radon measures on $Y$, where $Y$ is a Borel space. We have by definition that
\[\Lin(MY, \R) \cong \Borel(Y,U \R),\]
where $\R$ is the linear space corresponding to the separable Banach space $\R$ and $U\R$ is its underlying space. The right hand side is the set of bounded measurable maps from $Y$ to $\R$.

On the other hand, let $M_{\Ban} Y$ be the Banach space of Radon measures with the total variation norm. The space of arbitrary Banach space maps $l: M_{\Ban}Y \to \R$ contains many unpleasant characters. For example, let $g: Y \to \R$ be an arbitrary  bounded function, not necessarily measurable, and let $\R^{(Y)} \subset M_{\Ban} Y$ be the linear subspace consisting of finitely-supported measures on $Y$. For any measure $\mu$ in $\R^{(Y)}$, the integration
\[\sps{\sint g} (\mu) := \int g(y) d\mu(y)\]
is perfectly well-defined, because $\mu$ is finitely supported. Moreover, it is clear that $\sint g: \R^{(Y)} \to \R$ is a bounded linear fuctional, because $g$ is bounded. By the Hahn-Banach theorem, it extends to a bounded linear functional $\ti{\sint g}:M_{\Ban} Y\to \R$. Unless $g$ is chosen to be measurable,  the function $\ti {\sint g}$ does {\em not} correspond to a morphism in $\Lin(M Y,\R)$. In fact, composing with $\ti {\sint g}$ destroys the measurability of curves. We have
\[\sps{\ti {\sint g} }\cp \d(y) = g(y)\]
by construction, and $g: Y \to \R$ is not measurable. Essentially, we need to choose the weak topology on $MY$ instead of the strong topology, by {\em defining} the dual of $MY$ to be $\Borel(Y,U\R)$. However, the resulting locally convex space will not be bornological, and in particular we have left the setting of ``convenient'' spaces.
\section*{Acknowledgements}
The author is greatly indebted to Nir Gadish. This paper could not have been written without his encouragement and interest throughout its conception. 
\providecommand{\bysame}{\leavevmode\hbox to3em{\hrulefill}\thinspace}
\providecommand{\MR}{\relax\ifhmode\unskip\space\fi MR }
% \MRhref is called by the amsart/book/proc definition of \MR.
\providecommand{\MRhref}[2]{%
  \href{http://www.ams.org/mathscinet-getitem?mr=#1}{#2}
}
\providecommand{\href}[2]{#2}

\bibliographystyle{amsplain}

\begin{thebibliography}{10}

\bibitem{AdamekBorceuxLackRosicky2002}
Ji{\v r}\i\' Ad{\'a}mek, Francis Borceux, Stephen Lack, and Ji{\v r}\i\'
  Rosick{\'y}, \emph{A classification of accessible categories}, Journal of
  Pure and Applied Algebra \textbf{175} (2002), no.~1, 7--30.

\bibitem{ArtinGrothendieckVerdier1972}
Michael Artin, Alexander Grothendieck, and Jean-Louis Verdier (eds.),
  \emph{S{\'e}minaire de {{G{\'e}om{\'e}trie Alg{\'e}brique}} du {{Bois Marie}}
  - 1963-64 - {{Th{\'e}orie}} des topos et cohomologie {\'e}tale des
  sch{\'e}mas - ({{SGA}} 4)}, vol.~1, Lecture {{Notes}} in {{Mathematics}}, no.
  269, {Springer-Verlag}, {Berlin; New York}, 1972.

\bibitem{Aumann1961}
Robert~J. Aumann, \emph{Borel structures for function spaces}, Illinois Journal
  of Mathematics \textbf{5} (1961), 614--630.

\bibitem{BastianiEhresmann1972}
Andr{\'e}e Bastiani and Charles Ehresmann, \emph{Categories of sketched
  structures}, Cahiers de topologie et g{\'e}om{\'e}trie diff{\'e}rentielle
  cat{\'e}goriques \textbf{13} (1972), no.~2, 104--214.

\bibitem{Birkhoff1935}
Garrett Birkhoff, \emph{On the structure of abstract algebras}, Mathematical
  Proceedings of the Cambridge Philosophical Society \textbf{31} (1935), no.~4,
  433--454.

\bibitem{BorceuxDay1980}
Francis Borceux and Brian Day, \emph{Universal algebra in a closed category},
  Journal of Pure and Applied Algebra \textbf{16} (1980), no.~2, 133--147.

\bibitem{Day1970}
Brian Day, \emph{On closed categories of functors}, Reports of the {{Midwest
  Category Seminar}}, {{IV}}, Lecture {{Notes}} in {{Mathematics}}, {{Vol}}.
  137, {Springer, Berlin}, 1970, pp.~1--38.

\bibitem{Day1972}
\bysame, \emph{A reflection theorem for closed categories}, Journal of Pure and
  Applied Algebra \textbf{2} (1972), no.~1, 1--11.

\bibitem{Day1974}
\bysame, \emph{On closed categories of functors. {{II}}}, Category Seminar
  (Proc. Sem., Sydney, 1972/1973) (1974), 20--54. Lecture Notes in Math., Vol.
  420.

\bibitem{DayStreet1995}
Brian Day and Ross Street, \emph{Kan extensions along promonoidal functors},
  Theory and Applications of Categories \textbf{1} (1995), No.\textbackslash{}
  4, 72--77.

\bibitem{Dynkin1961}
E.~B. Dynkin, \emph{Theory of {{Markov}} processes}, Translated from the
  {{Russian}} by {{D}}. {{E}}. {{Brown}}; Edited by {{T}}. {{K{\"o}v{\'a}ry}},
  {Prentice-Hall, Inc., Englewood Cliffs, N.J.; Pergamon Press,
  Oxford-London-Paris}, 1961.

\bibitem{Ehresmann1968}
Charles Ehresmann, \emph{Esquisses et types des structures alg{\'e}briques},
  Bul. Inst. Politehn. Ia{\c s}i (N.S.) \textbf{14 (18)} (1968), no.~fasc. 1-2,
  1--14.

\bibitem{FrolicherKriegl1988}
Alfred Fr{\"o}licher and Andreas Kriegl, \emph{Linear spaces and
  differentiation theory}, Pure and {{Applied Mathematics}} ({{New York}}),
  {John Wiley \& Sons, Ltd., Chichester}, 1988, A Wiley-Interscience
  Publication.

\bibitem{GabrielUlmer1971}
Peter Gabriel and Friedrich Ulmer, \emph{Lokal pr{\"a}sentierbare
  {{Kategorien}}}, Lecture {{Notes}} in {{Mathematics}}, {{Vol}}. 221,
  {Springer-Verlag, Berlin-New York}, 1971.

\bibitem{Giraud1971}
Jean Giraud, \emph{Cohomologie non ab{\'e}lienne}, {Springer-Verlag, Berlin-New
  York}, 1971.

\bibitem{Giry1982}
Mich{\`e}le Giry, \emph{A categorical approach to probability theory},
  Categorical Aspects of Topology and Analysis ({{Ottawa}}, {{Ont}}., 1980),
  Lecture {{Notes}} in {{Math}}., vol. 915, {Springer, Berlin-New York}, 1982,
  pp.~68--85.

\bibitem{HeunenKammarStatonYang2017}
C.~Heunen, O.~Kammar, S.~Staton, and H.~Yang, \emph{A convenient category for
  higher-order probability theory}, 2017 32nd {{Annual ACM}}/{{IEEE Symposium}}
  on {{Logic}} in {{Computer Science}} ({{LICS}}), June 2017, pp.~1--12.

\bibitem{JohnstoneWraith1978}
Peter~T. Johnstone and Gavin~C. Wraith, \emph{Algebraic theories in toposes},
  Indexed Categories and Their Applications, Lecture {{Notes}} in {{Math}}.,
  vol. 661, {Springer, Berlin}, 1978, pp.~141--242.

\bibitem{Kelly1974}
G.~M. Kelly, \emph{Doctrinal adjunction}, Category Seminar (Proc. Sem., Sydney,
  1972/1973) (1974), 257--280. Lecture Notes in Math., Vol. 420.

\bibitem{Kelly1982a}
\bysame, \emph{Structures defined by finite limits in the enriched context.
  {{I}}}, Cahiers de Topologie et G{\'e}om{\'e}trie Diff{\'e}rentielle
  \textbf{23} (1982), no.~1, 3--42.

\bibitem{Kelly1982b}
Gregory~Maxwell Kelly, \emph{Basic concepts of enriched category theory},
  London {{Mathematical Society Lecture Note Series}}, vol.~64, {Cambridge
  University Press, Cambridge-New York}, 1982.

\bibitem{Kennison1968}
J.~F. Kennison, \emph{On limit-preserving functors}, Illinois Journal of
  Mathematics \textbf{12} (1968), no.~4, 616--619.

\bibitem{Kuratowski1948}
Casimir Kuratowski, \emph{Topologie. {{I}}. {{Espaces M{\'e}trisables}},
  {{Espaces Complets}}}, Monografie {{Matematyczne}}, Vol. 20,
  {Warszawa-Wroc{\l}aw}, 1948.

\bibitem{LackRosicky2011}
Stephen Lack and Ji{\v r}{\'i} Rosick\'{y}, \emph{Notions of {{Lawvere}}
  theory}, Applied Categorical Structures \textbf{19} (2011), no.~1, 363--391.

\bibitem{Lawvere1963}
F.~William Lawvere, \emph{Functorial semantics of algebraic theories},
  Proceedings of the National Academy of Sciences \textbf{50} (1963), no.~5,
  869--872.

\bibitem{Linton1966}
F.~E.~J. Linton, \emph{Some aspects of equational categories}, Proceedings of
  the {{Conference}} on {{Categorical Algebra}}, {Springer, Berlin,
  Heidelberg}, 1966, pp.~84--94 (en).

\bibitem{Linton1969}
\bysame, \emph{An outline of functorial semantics}, Seminar on {{Triples}} and
  {{Categorical Homology Theory}}, Lecture {{Notes}} in {{Mathematics}},
  {Springer, Berlin, Heidelberg}, 1969, pp.~7--52 (en).

\bibitem{Lucyshyn-Wright2016}
Rory B.~B. {Lucyshyn-Wright}, \emph{Enriched algebraic theories and monads for
  a system of arities}, Theory and Applications of Categories \textbf{31}
  (2016), Paper No. 5, 101--137.

\bibitem{Mackey1957}
George~W. Mackey, \emph{Borel structure in groups and their duals},
  Transactions of the American Mathematical Society \textbf{85} (1957), no.~1,
  134--165.

\bibitem{Power1999}
John Power, \emph{Enriched {{Lawvere}} theories}, Theory and Applications of
  Categories \textbf{6} (1999), 83--93.

\bibitem{Shulman2008}
Michael Shulman, \emph{Framed bicategories and monoidal fibrations}, Theory and
  Applications of Categories \textbf{20} (2008), No. 18, 650--738.

\bibitem{Shulman2013}
\bysame, \emph{Enriched indexed categories}, Theory and Applications of
  Categories \textbf{28} (2013), 616--696.

\end{thebibliography}
\end{document}